\documentclass[12pt]{amsart}

\usepackage{times}
\usepackage{amssymb}
\usepackage{graphicx,xspace}
\usepackage{epsfig}
\usepackage{enumitem}
\usepackage{xcolor}

\usepackage[T1]{fontenc}
\usepackage[utf8]{inputenc}
\usepackage{youngtab}

\usepackage{hyperref}

\theoremstyle{plain}
\newtheorem{lemma}{Lemma}[section]
\newtheorem{theorem}[lemma]{Theorem}
\newtheorem{corollary}[lemma]{Corollary}
\newtheorem{proposition}[lemma]{Proposition}

\newtheorem{conjecture}[lemma]{Conjecture}

\theoremstyle{remark}
\newtheorem{remark}[lemma]{Remark}
\newtheorem*{example}{Example}
\newtheorem{definition}[lemma]{Definition}

\newcommand{\R}{\mathbb{R}}
\newcommand{\C}{\mathbb{C}}
\newcommand{\N}{\mathbb{N}}

\newcommand{\p}{\mathfrak{p}}

\newcommand{\Sym}[1]{\mathfrak{S}_{#1}}

\newcommand{\M}{\mathcal{M}}

\newcommand{\Z}{\mathbb{Z}}

\newcommand{\loops}{\mathcal{L}}
\newcommand{\pp}{\mathbf{p}}
\newcommand{\qq}{\mathbf{q}}

\newcommand{\question}[1]{\textcolor{red}{\emph{#1}}}
\newcommand{\orient}{\phi}
\newcommand{\pairings}{\mathcal{P}}

\DeclareMathOperator{\Tr}{Tr}

\DeclareMathOperator{\Gl}{GL}
\DeclareMathOperator{\Aut}{Aut}
\DeclareMathOperator{\Sp}{Sp}
\DeclareMathOperator{\type}{type}
\DeclareMathOperator{\Stab}{Stab}

\author{Valentin F\'eray}
\address{LaBRI, Universit\'e Bordeaux 1, 351 cours de la Lib\'eration, 33 400
Talence, France}
\email{feray@labri.fr}

\author{Piotr \'Sniady}
\address{Institute of Mathematics, Polish Academy of Sciences,
\mbox{ul.~Śniadec\-kich 8,} 00-956 Warszawa, Poland \newline
\indent Institute of Mathematics,
University of Wroclaw,  \mbox{pl.\ Grunwaldzki~2/4,} 50-384
Wroclaw, Poland}
\email{Piotr.Sniady@math.uni.wroc.pl}

\title[Zonal polynomials]{Zonal polynomials via Stanley's coordinates\\ and
free cumulants}

\begin{document}

\begin{abstract}
We study zonal characters which are defined as suitably normalized coefficients
in the expansion of zonal polynomials in terms of power-sum symmetric functions.
We show that the zonal characters, just like the characters of the symmetric
groups,
admit a nice combinatorial des\-cription in terms of Stanley's multirectangular
coordinates of Young dia\-grams.
We also study the analogue of Kerov polynomials,
namely we express the zonal characters as
polynomials in free cumulants and we give an explicit combinatorial
interpretation of their coefficients. In this way, we prove
two recent conjectures of Lassalle for Jack polynomials in
the special case of zonal polynomials.
\end{abstract}

\maketitle

\section{Introduction}

\subsection{Zonal polynomials}

\subsubsection{Background}
Zonal polynomials were introduced by Hua \cite[Chapter VI]{Hua1963} and later studied
by James \cite{James1960,James1961} in order to solve some problems from statistics
and multivariate analysis. They quickly became a fundamental tool in this
theory as well as in the random matrix theory (an overview can be found in the
book of Muirhead \cite{Muirhead1982} or also in the introduction to the
monograph of Takemura \cite{Takemura1984}). They also appear in
the representation theory of the Gelfand pairs
$(\Sym{2n},H_n)$ (where $\Sym{2n}$ and $H_n$ are, respectively,
the symmetric and hyperoctahedral groups) and $(\Gl_d(\R),
O_d)$.
More precisely, when we expand zonal polynomials in the power-sum basis of the symmetric function ring,
the coefficients describe a canonical basis ({\it i.e.} the zonal spherical functions)
of the algebra of left and right $H_n$-invariant (resp.~$O_d$-invariant) functions on $\Sym{2n}$
(resp.~$\Gl_d(\R)$).

This last property shows that zonal polynomials can be viewed as an analogue
of Schur symmetric functions: the latter appear when we look at left and right
$\Sym{n}$ (resp. $U_d$) invariant functions on $\Sym{n} \times \Sym{n}$ (resp. $\Gl_d(\C)$).
the Gelfand pairs $(\Sym{n} \times \Sym{n}, \Sym{n})$ and $(\Gl_d(\C), U_d)$.
This is the underlying principle why
many of the properties of Schur functions can be extended to zonal polynomials
and this article goes in this direction. 

In this article we use a characterization of zonal polynomials 
due to James \cite{James1961} as their definition. The elements needed in our development
(including the precise definition of zonal polynomials) are given in Section
\ref{subsec:definition}. For a more complete introduction to the topic
we refer to the Chapter VII of Macdonald's book \cite{Macdonald1995}.

The main results of this article are new combinatorial formulas for zonal polynomials.
Note that, as they are a particular case of Jack symmetric
functions, there exists already a combinatorial interpretation for them in terms
of ribbon tableaux (due to Stanley \cite{Stanley1989}). But our formula is of
different type: it gives a combinatorial interpretation to the coefficients of
the zonal polynomial $Z_\lambda$ expanded in the power-sum basis as a function
of $\lambda$. In more concrete words, the combinatorial objects describing the
coefficient of $p_\mu$ in $Z_\lambda$ depend on $\mu$, whereas the statistics on
them depend on $\lambda$ (in Stanley's result it is roughly the opposite). This
kind of \emph{dual} approach makes appear shifted symmetric functions
\cite{OkounkovOlshanski1997} and is an analogue of recent developments
concerning characters of the symmetric group: more details will be given in
Section \ref{subsect:characters}.

      \subsubsection{Jack polynomials}
Jack \cite{Jack1970/1971} introduced a family of symmetric functions
$J^{(\alpha)}_\lambda$ depending on an additional parameter
$\alpha$. These functions are now called \emph{Jack polynomials}.
For some special values of $\alpha$ they coincide with some
established families of symmetric functions. Namely, up to multiplicative
constants, for $\alpha=1$ Jack polynomials coincide with Schur polynomials, for
$\alpha=2$ they coincide with zonal polynomials, for $\alpha=\frac{1}{2}$ they
coincide with symplectic zonal polynomials, for $\alpha = 0$ we recover the
elementary symmetric functions and finally their highest degree component in
$\alpha$ are the monomial symmetric functions. Moreover, some other
specializations appear in different contexts: the case $\alpha = 1/k$, where $k$
is an integer, has been considered by Kadell in relation with generalizations of
Selberg's integral \cite{Kadell1997}. In addition, Jack polynomials for
$\alpha=-(k+ 1)/(r+ 1)$ verify some interesting annihilation conditions
\cite{Feigin2002}.

Jack polynomials for a generic value of the parameter $\alpha$  do not seem
to have a direct interpretation, for example in
the context of the representation theory or in the theory of zonal spherical 
functions of some Gelfand pairs.
Nevertheless, over the time it has been shown that several
results concerning Schur and zonal polynomials can be generalized in a rather
natural way to Jack polynomials (see, for example, the work of Stanley
\cite{Stanley1989}), therefore Jack polynomials can be viewed as a natural
interpolation between several interesting families of symmetric functions at the
same time.

An extensive numerical exploration and conjectures done by Lassalle
\cite{Lassalle2008a,Lassalle2009} suggest that the kind of combinatorial
formulas we establish in this paper has generalizations for any value of the
parameter $\alpha$. Unfortunately, we are not yet able to achieve this goal.

   \subsection{The main result 1: a new formula for zonal polynomials}

      \label{subsubsec:main_result}

      \subsubsection{Pair-partitions}

The central combinatorial objects in this paper are pair-partitions:
\begin{definition}
A pair-partition $P$ of $[2n]=\{1,\ldots,2n\}$ is a set of pairwise disjoint
two-element sets, such that their (disjoint) union is equal to
$[2n]$. A pair-partition can be seen as an involution of
$[2n]$ without fixpoints, which associates to each element its
partner from the pair.
\end{definition}
The simplest example is the \textit{first} pair-partition, which will play a
par\-ti\-cu\-lar role in our article:
\begin{equation}
    \label{eq:first-pair-partition}
    S=\big\{ \{1,2\},\{3,4\},\dots,\{2n-1,2n\}\big\}. 
\end{equation}

\subsubsection{Couple of pair-partitions}
\label{SubsubsectCouplePP}
Let us consider two pair-partitions $S_1,S_2$ of the same set $[2n]$.
We consider the following bipartite edge-labeled graph $\loops(S_1,S_2)$:
\begin{itemize}
    \item it has $n$ black vertices indexed by the two-element sets of $S_1$
        and $n$ white vertices indexed by the two-element sets $S_2$;
    \item its edges are labeled with integers from $[2n]$.
        The extremities of the edge labeled $i$ are the two-element sets of $S_1$
        and $S_2$ containing $i$.
\end{itemize}
Note that each vertex has degree $2$ and each edge has one white and one black
extremity. Besides, if we erase the indices of the vertices, it is easy to
recover them from the labels of the edges (the index of a vertex is the set of
the two labels of the edges leaving this vertex). Thus, we forget the indices
of the vertices and view $\loops(S_1,S_2)$ as an edge-labeled graph.

As every vertex has degree $2$, the graph $\loops(S_1,S_2)$ is a collection of
loops. Moreover, because of the proper bicoloration of the vertices, all loops
have even length. Let $2\ell_1\geq 2\ell_2\geq\cdots$ be the ordered lengths of
these loops. The partition $(\ell_1,\ell_2,\dots)$ is called the type of
$\loops(S_1,S_2)$ or the type of the couple $(S_1,S_2)$. Its length, {\it
i.e.~}the number of connected components of the graph $\loops(S_1,S_2)$, will be
denoted by $|\loops(S_1,S_2)|$ (we like to see $\loops(S_1,S_2)$ as a set of
loops). We define the sign of a couple of pair-partitions as follows:
$$ (-1)^{\loops(S_1,S_2)} = (-1)^{(\ell_1-1)+(\ell_2-1)+\cdots} = (-1)^{n -
|\loops(S_1,S_2)|}$$
and the power-sum symmetric function
\begin{equation}
\label{eq:power-sum}
 p_{\loops(S_1,S_2)}(z_1,z_2,\dots) = p_{\ell_1,\ell_2,\dots}(z_1,z_2,\dots)
= \prod_{i} \sum_j z_j^{\ell_i}. 
\end{equation}

\begin{minipage}[t]{.4\linewidth}
{\it Example.} We consider
\[\begin{array}{c}
    S_1 = \big\{ \{1,2\},\{3,4\},\{5,6\} \big\}; \\
    S_2 = \big\{ \{1,3\},\{2,4\},\{5,6\} \big\}. \\
\end{array}\]
\end{minipage}
\begin{minipage}[t]{.6\linewidth}
    \[\text{Then }\loops(S_1,S_2) = \begin{array}{c}
        \includegraphics[height=1.2cm]{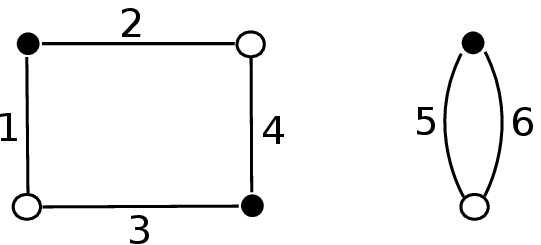}
    \end{array}.\]
\end{minipage}
So, in this case, $\loops(S_1,S_2)$ has type $(2,1)$.

Another, more complicated, example is given in the beginning of Section
\ref{SubsectGluings}.

\subsubsection{Zonal polynomials and pair-partitions}

For zonal and Jack polynomials we use in this article the notation from
Macdonald's book \cite{Macdonald1995}.
In particular, the zonal polynomial $Z_\lambda$ associated to the partition
$\lambda$ is the symmetric function defined by Eq.~(2.13) of
\cite[VII.2]{Macdonald1995}.
For the reader not accustomed with zonal polynomials, their property given in
Section \ref{subsec:definition} entirely determines them and
is the only one used in this paper.

Let $\lambda=(\lambda_1,\lambda_2,\dots)$ be a partition of $n$; we consider the
Young tableau $T$ of shape $2\lambda=(2\lambda_1,2\lambda_2,\dots)$ in which the
boxes are numbered consecutively along the rows. Permutations of $[2n]$ can be
viewed as permutations of the boxes of $T$. Then a pair $(S_1,S_2)$ is called
\emph{$T$-admissible} if $S_1,S_2$ are pair-partitions of $[2n]$ such that
$S\circ S_1$ preserves each column of $T$ and $S_2$ preserves each row.

\begin{theorem}\label{theo:zonal-polynomials}
With the definitions above, the zonal polynomial is given by
$$Z_\lambda = \sum_{(S_1,S_2)\ T\text{-admissible}} (-1)^{\loops(S,S_1)}\
p_{\loops(S_1,S_2)}.$$
\end{theorem}
This result will be proved in Section \ref{subsec:proof-of-theo-zonal}.

\begin{example}
Let $\lambda=(2,1)$ and $T=\begin{array}{c} \young(1234,56) \end{array}$.
Then $(S_1,S_2)$ is $T$-admissible if and only if:
    \begin{multline*}S_1 \in \bigg\{ \big\{ \{1,2\},\{3,4\},\{5,6\} \big\}, 
                        \big\{ \{1,6\},\{3,4\},\{2,5\} \big\}
			\bigg\}
\text{ and} \\
S_2 \in \bigg\{ \big\{
\{1,2\},\{3,4\},\{5,6\} \big\}, 
                        \big\{ \{1,3\},\{2,4\},\{5,6\} \big\}, \\
                        \big\{ \{1,4\},\{2,3\},\{5,6\} \big\} \bigg\}. 
  \end{multline*}
The first possible value of $S_1$ gives $(-1)^{\loops(S,S_1)}=1$ and the correspon\-ding types
of $\loops(S_1,S_2)$ for the three possible values of $S_2$ are, respectively,
$(1,1,1)$, $(2,1)$ and $(2,1)$. For the second value of $S_1$, the sign is given by
$(-1)^{\loops(S,S_1)}=-1$ and the types of the corresponding set-partitions
$\loops(S_1,S_2)$ are, respectively, $(2,1)$, $(3)$ and $(3)$.

Finally, one obtains $Z_{(2,1)} = p_{(1,1,1)} + p_{(2,1)} - 2 p_{(3)}$.
\end{example}

\begin{remark}\label{rem:analogue_schur}
    This theorem is an analogue of a known result on Schur symmetric
    functions:
    \begin{displaymath}
        \frac{n! \cdot s_\lambda}{\dim(\lambda)}=
	  \sum (-1)^{\sigma_1}\ p_{\type(\sigma_1 \circ \sigma_2)},
    \end{displaymath}
    where the sum runs over pairs of permutations $(\sigma_1,\sigma_2)$ of the
    boxes of the diagram $\lambda$ such that $\sigma_1$ (resp.~$\sigma_2$)
    preserves the columns (resp.~the rows) of $\lambda$ and $\type(\sigma_1
    \circ \sigma_2)$ denotes the partition describing the lengths of the cycles
    of $\sigma_1 \circ \sigma_2$. This formula is a consequence of the 
    explicit construction of the representation associated to $\lambda$
    via the Young symmetrizer. For a detailed proof, see 
    \cite{F'eray'Sniady-preprint2007}.
    In \cite{Hanlon1988}, the author tries unsuccessfully to generalize it
    to Jack polynomials by introducing some statistics on couples of 
    permutations.
    Our result shows that, at least for $\alpha=2$, a natural way to
    generalize is to use other
    combinatorial objects than permutations. 
\end{remark}

   \subsection{Zonal characters}
\label{subsect:characters}
The above formula expresses zonal polynomials in terms of power-sum
symmetric functions. In Section \ref{sec:Stanley}, we will extract the
coefficient of a given power-sum. In this way we study an analogue
of the coordinates of Schur polynomials in the power-sum basis
of the symmetric function ring. These coordinates are known
to be the irreducible characters of the symmetric group
and have a plenty of interesting properties. Some of them are (conjecturally)
generalizable to the context where Schur functions are replaced by Jack
polynomials and our results in the case of zonal polynomials go in this
direction.

      \subsubsection{Characters of symmetric groups}
For a Young diagram $\lambda$ we denote by $\rho^\lambda$ the corresponding
irreducible representation of the symmetric group $\Sym{n}$ with $n=|\lambda|$.
Any partition $\mu$ such that $|\mu|=n$ can be viewed as a conjugacy class in
$\Sym{n}$. Let $\pi_\mu\in\Sym{n}$ be any permutation from this conjugacy
class; we will denote by $\Tr \rho^\lambda(\mu):=\Tr \rho^\lambda(\pi_\mu)$ the
corresponding irreducible character value. If $m\leq n$, any permutation
$\pi\in\Sym{m}$ can be also
viewed as an element of $\Sym{n}$, we just have to add $n-m$ additional
fixpoints to $\pi$; for this reason
$$ \Tr \rho^\lambda(\mu) := \Tr \rho^\lambda\left(\mu\
1^{|\lambda|-|\mu|}\right) $$
makes sense also when $|\mu| \leq |\lambda|$.

Normalized characters of the symmetric group were defined by Ivanov and Kerov
\cite{IvanovKerov1999} as follows:
\begin{equation}
\label{eq:ivanov-kerov}
\Sigma_{\mu}^{(1)}(\lambda)= 
\underbrace{n (n-1) \cdots (n-|\mu|+1)}_{\text{$|\mu|$ factors}} 
\frac{\Tr \rho^\lambda(\mu)}{\text{dimension of $\rho^\lambda$}}
\end{equation}
(the meaning of the superscript in the notation $\Sigma_{\mu}^{(1)}(\lambda)$
will become clear later on). The novelty of the idea was to view the character
as a function $\lambda\mapsto\Sigma^{(1)}_{\mu}(\lambda)$ on the set of Young
diagrams (of any size) and to keep the conjugacy class fixed. The normalization
constants in \eqref{eq:ivanov-kerov} were chosen in such a way that the normalized
characters $\lambda\mapsto\Sigma^{(1)}_{\mu}(\lambda)$ form a linear basis (when
$\mu$ runs over the set of all partitions) of the algebra $\Lambda^\star$ of
shifted symmetric functions introduced by Okounkov and Olshanski
\cite{OkounkovOlshanski1997}, which is very rich in structure (this property is,
for example, the key point in a recent approach to study asymptotics of random
Young diagrams under Plancherel measure \cite{IvanovOlshanski2002}). In
addition, recently a combinatorial description of the quantity
\eqref{eq:ivanov-kerov} has been given \cite{Stanley-preprint2006,F'eray2010},
which is particularly suitable for study of asymptotics of character values
\cite{F'eray'Sniady-preprint2007}.

Thanks to Frobenius' formula for characters of the symmetric groups
\cite{Frobenius1900},
definition \eqref{eq:ivanov-kerov} can be rephrased using Schur functions.
We expand the Schur polynomial $s_\lambda$ in the base of the power-sum
symmetric functions $(p_\rho)$ as follows:
\begin{equation} 
\label{eq:schur}
\frac{n!\ s_\lambda}{\dim(\lambda)}=\sum_{\substack{\rho: \\ |\rho|=|\lambda|}} 
\theta^{(1)}_{\rho}(\lambda) \ p_{\rho}
\end{equation}
for some numbers $\theta^{(1)}_{\rho}(\lambda)$. Then
\begin{equation}
\label{eq:lassalle} 
\Sigma_{\mu}^{(1)}(\lambda)=\binom{|\lambda|-|\mu|+m_1(\mu)}{m_1(\mu)}
\ z_\mu \ \theta^{(1)}_{\mu,1^{|\lambda|-|\mu|}}(\lambda), 
\end{equation}
where
$$z_\mu = \mu_1 \mu_2 \cdots \ m_1(\mu)!\ m_2(\mu)! \cdots$$
and $m_i(\mu)$ denotes the multiplicity of $i$ in the partition $\mu$.

      \subsubsection{Zonal and Jack characters}

In this paragraph we will define analogues of the quantity
$\Sigma_{\mu}^{(1)}(\lambda)$ via Jack polynomials. First of all, as there are
several of them, we have to fix a normalization for Jack polynomials. In our
context, the best is to use the functions denoted by $J$ in the book of
Macdonald \cite[VI, (10.22)]{Macdonald1995}. With this normalization, one has
\begin{align*}
 J^{(1)}_\lambda&=\frac{n!\ s_\lambda}{\dim(\lambda)}, \\
 J^{(2)}_\lambda&=Z_\lambda.
\end{align*}

If in \eqref{eq:schur}, we replace the left-hand side by Jack polynomials:
\begin{equation} 
\label{eq:jack-characters}
J_\lambda^{(\alpha)}=\sum_{\substack{\rho: \\
|\rho|=|\lambda|}} 
\theta_{\rho}^{(\alpha)}(\lambda)\ p_{\rho} 
\end{equation}
then in analogy to \eqref{eq:lassalle}, we define
\begin{displaymath}
\Sigma_{\mu}^{(\alpha)}(\lambda)=
\binom{|\lambda|-|\mu|+m_1(\mu)}{ m_1(\mu) }
\ z_\mu \ \theta^{(\alpha)}_{\mu,1^{|\lambda|-|\mu|}}(\lambda).
\end{displaymath}

These quantities are called \emph{Jack characters}.
Notice that for $\alpha=1$ we recover the usual
normalized character values of the symmetric groups. The case $\alpha=2$ is
of central interest in this article, since then the left-hand side of
\eqref{eq:jack-characters} is equal to the zonal polynomial; for this reason
$\Sigma_{\mu}^{(2)}(\lambda)$ will be called \emph{zonal character}.

Study of Jack characters has been initiated by Lassalle
\cite{Lassalle2008a,Lassalle2009}. Just like the usual normalized characters
$\Sigma_{\mu}^{(1)}$, they are ($\alpha$-)shifted symmetric functions 
\cite[Proposition 2]{Lassalle2008a} as well, which is a good hint that they
might be an interesting generalization of the characters.
The names \emph{zonal characters} and
\emph{Jack characters} are new; we decided to introduce them because quantities
$\Sigma_{\mu}^{(\alpha)}(\lambda)$ are so interesting that they deserve a
separate name. One could argue that this name is not perfect since
Jack characters are not \textit{sensu stricto} characters in the sense of the
representation theory (as opposed to, say, zonal characters which are closely
related to the zonal spherical functions and therefore are a natural extension
of the characters in the context of Gelfand pairs).
On the other hand, as we shall see, Jack
characters conjecturally share many interesting properties with the usual 
and zonal characters of
symmetric groups, therefore the former can be viewed as interpolation of the
latter which justifies to some extent their new name.

\subsection{The main result 2: combinatorial formulas for zonal characters}

      \subsubsection{Zonal characters in terms of numbers of colorings
                     functions}

Let $S_0$, $S_1$, $S_2$ be three pair-partitions of the set $[2k]$.
We consider the following function on the set of Young diagrams:

\begin{definition}
    Let $\lambda$ be a partition of any size. We define
    $N^{(1)}_{S_0,S_1,S_2}(\lambda)$ as the number of functions $f$ from
    $[2k]$ to the boxes of the Young diagram $\lambda$ such that
    for every $l\in[2k]$:
    \begin{enumerate}[label=(Q\arabic*)]
        \addtocounter{enumi}{-1}
        \item
            \label{cond:Q0}
            $f(l)=f(S_0(l))$, in other words $f$ can be viewed as a function on
            the set of pairs constituting $S_0$;
        \item
            \label{cond:Q1}
            $f(l)$ and $f(S_1(l))$ are in the same column;
        \item
            \label{cond:Q2} $f(l)$ and $f(S_2(l))$ are in the same row.
    \end{enumerate}
    \label{def:Na}
\end{definition}

Note that $\lambda\mapsto N^{(1)}_{S_0,S_1,S_2}(\lambda)$ is, in general, not
a shifted symmetric function, so it cannot be expressed in terms of  zonal
characters. On the other hand, the zonal characters have a
very nice expression in terms of the $N$ functions:
\begin{theorem}\label{theo:ZonalFeraySniady-A1}
    Let $\mu$ be a partition of the integer $k$ and $(S_1,S_2)$ be a fixed
    couple of pair-partitions of the set $[2k]$ of type $\mu$.
    Then one has the following equality between functions on the set of Young
    diagrams:
    \begin{equation}
        \Sigma^{(2)}_\mu =\frac{1}{2^{\ell(\mu)}}
                \sum_{S_0} (-1)^{\loops(S_0,S_1)}\
                  2^{|\loops(S_0,S_1)|}\ N^{(1)}_{S_0,S_1,S_2},
        \label{eq:ZonalFeraySniady-A1}
    \end{equation}
    where the sum runs over pair-partitions of\/ $[2k]$ and
    $\ell(\mu)$ denotes the number of parts of partition $\mu$.
\end{theorem}

We postpone the proof to Sections
\ref{subsec:reformulation}--\ref{subsec:forgetting-injectivity}.
This formula is an intermediate step towards Theorem \ref{theo:ZonalFeraySniady-B},
but we wanted to state it as an independent result because its analogue for the usual
characters \cite[Theorem 2]{F'eray'Sniady-preprint2007} has been quite useful in
some contexts (see \cite{F'eray'Sniady-preprint2007,F'eray2009}).

\begin{example}
    Let us consider the case $\mu=(2)$.
    We fix $S_1=\big\{ \{1,2\},\{3,4\} \big\}$ and 
    $S_2=\big\{ \{1,4\},\{2,3\} \big\}$.
    Then $S_0$ can take three possible values: $S_1$, $S_2$ and 
    $S_3:= \big\{ \{1,3\},\{2,4\} \big\}$.

    If $S_0=S_1$, condition \ref{cond:Q0} implies condition \ref{cond:Q1}.
    Moreover, conditions \ref{cond:Q0} and \ref{cond:Q2} imply that the images
    of all elements are in the same row.
    Therefore $N^{(1)}_{S_1,S_1,S_2}(\lambda)$ is equal to the number of ways
    to choose two boxes in the same row of $\lambda$: one is the image of $1$ 
    and $2$ and the other the image of $3$ and $4$. It follows that
    \[N^{(1)}_{S_1,S_1,S_2}(\lambda) = \sum_i \lambda_i^2.\]

    In a similar way, $N^{(1)}_{S_2,S_1,S_2}(\lambda)$ is the number of ways
    to choose two boxes in the same column of $\lambda$: one is the image of $1$
    and $4$ and the other the image of $2$ and $3$. It follows that
    \[N^{(1)}_{S_2,S_1,S_2}(\lambda) = \sum_i (\lambda'_i)^2,\]
    where $\lambda'$ is the conjugate partition of $\lambda$.

    Consider the last case $S_0=S_3$. 
    Conditions \ref{cond:Q0} and \ref{cond:Q2} imply that the images
    of all elements are in the same row. 
    Besides, conditions \ref{cond:Q0} and \ref{cond:Q1} imply that the images
    of all elements are in the same column.
    So all elements must be matched to the same box and the number of functions
    fulfilling the three properties is simply the number of boxes of $\lambda$.

    Finally,
    \begin{equation}
        \Sigma^{(2)}_{(2)}(\lambda) = 2 \left( \sum_i \lambda_i^2 \right)
            - \left( \sum_i (\lambda'_i)^2 \right) - |\lambda|.
        \label{EqExSigmaN}
    \end{equation}
    If we denote $n(\lambda)= \sum_i \binom{\lambda'_i}{2}$
    \cite[equation (I.1.6)]{Macdonald1995}, this can be rewritten as:
    \begin{align*}
        \Sigma^{(2)}_{(2)}(\lambda) &= 2 (2n(\lambda') + |\lambda|)
            - (2n(\lambda) + |\lambda|) - |\lambda|  \\
	    &= 4 n(\lambda') - 2 n(\lambda).
    \end{align*}
    The last equation corresponds to the case $\alpha=2$ of Example 1b.~of 
    paragraph VI.10 of Macdonald's book \cite{Macdonald1995}.
\end{example}

          \subsubsection{Zonal characters in terms of Stanley's coordinates}
          \label{SubsectZonalCharPQ}

The notion of Stanley's coordinates was introduced by Stanley
\cite{Stanley2003/04} who found a nice formula for normalized irreducible character
values of the symmetric group corresponding to rectangular Young diagrams. In
order to generalize this result, he defined, given two sequences $\pp$ and $\qq$
of positive integers of same size ($\qq$ being non-increasing), the partition:
$$ \pp\times \qq = (\underbrace{q_1,\dots,q_1}_\text{$p_1$ times},\dots,
\underbrace{q_l,\dots,q_l}_\text{$p_l$ times} ).$$
Then he suggested to consider the quantity $\Sigma^{(1)}_\mu(\pp \times \qq)$
as a polynomial in $\pp$ and $\qq$. An explicit combinatorial interpretation
of the coefficients was conjectured in \cite{Stanley-preprint2006} and proved
in \cite{F'eray2010}.


It is easy to deduce from the expansion of $\Sigma^{(2)}_\mu$ in terms of the
$N$ functions a combinatorial description of the polynomial
$\Sigma^{(2)}_\mu(\pp \times \qq)$.

\begin{theorem} \label{theo:ZonalFeraySniady-B}
    Let $\mu$ be a partition of the integer $k$ and $(S_1, S_2)$ be a fixed
    couple of pair-partitions of\/ $[2k]$ of type $\mu$.
    Then:
    \begin{multline}
        \Sigma^{(2)}_\mu(\pp \times \qq) = \\ \frac{(-1)^k}{2^{\ell(\mu)}}
 \sum_{S_0} \left[
\sum_{\phi: \loops(S_0,S_2) \to \N^\star}\
\prod_{l\in\loops(S_0,S_2)}(p_{\varphi(l)}) \cdot
\prod_{l'\in \loops(S_0,S_1)}(- 2q_{\psi(l')})
\right],
        \label{eq:ZonalFeraySniady-B}
    \end{multline}
    where $\psi(l'):=\max\limits_l \varphi(w)$ with $l$ running over the loops of
    $\loops(S_0,S_1)$
    having at least one element in common with $l'$.
\end{theorem}
We postpone the proof until Section \ref{subsec:N-in-terms}.

\begin{example}
\label{ex:example}
    We continue the previous example in the case $\mu=(2)$.

    When $S_0=S_1$, the graph $\loops(S_0,S_2)$ has only one loop, thus we sum
    over index $i \in \N^\star$.
    The graph $\loops(S_0,S_1)$ has two loops in this case, whose images by
    $\psi$ are both $i$.
    So the expression in the square brackets for $S_0=S_1$ is equal to:
    \[4 \sum_i p_i q_i^2.\]

    When $S_0=S_2$, the graph $\loops(S_0,S_2)$ has two loops, thus we sum over
    couples $(i,j)$ in $(\N^\star)^2$.
    The graph $\loops(S_0,S_1)$ has only one loop, which has elements in common
    with both loops of $\loops(S_0,S_2)$ and thus its image by $\psi$ is
    $\max(i,j)$.
    Therefore, the expression in the brackets can be written in this case as:
    \[-2 \sum_{i,j} p_i p_j q_{\max(i,j)}.\]

    When $S_0=S_3$, both graphs $\loops(S_0,S_2)$ and $\loops(S_0,S_1)$ have
    only one loop.
    Thus we sum over one index $i \in \N^\star$ which is the image by $\varphi$
    and $\psi$ of these loops.
    In this case the expression in the brackets is simply equal to:
    \[-2 \sum_{i,j} p_i q_i. \]

    Finally, in this case, Eq.~\eqref{eq:ZonalFeraySniady-B} becomes:
    \[ \Sigma^{(2)}_{(2)}(\pp \times \qq) = 2 \sum_i p_i q_i^2
        - \sum_{i,j} p_i p_j q_{\max(i,j)} - \sum_i p_i q_i. \]
    It matches the numerical data given by M. Lassalle in 
    \cite[top of page 3]{Lassalle2008a} (one has to change the signs and
    substitute $\beta=1$ in his formula).
\end{example}

            \subsection{Kerov polynomials}

      \subsubsection{Free cumulants}
For a Young diagram $\lambda=(\lambda_1,\lambda_2,\dots)$ and an integer
$s\geq 1$ we consider the dilated Young diagram
$$ D_s \lambda = (\underbrace{s\lambda_1,\dots,s\lambda_1}_{\text{$s$ times}},
\underbrace{s\lambda_2,\dots,s\lambda_2}_{\text{$s$ times}},\dots).$$
If we interpret the Young diagrams geometrically as collections of boxes then
the dilated diagram $D_s \lambda$ is just the image of $\lambda$ under scaling
by factor $s$.

This should not be confused with
$$\alpha \lambda = (\alpha \lambda_1,\alpha \lambda_2,\dots) $$
which is the Young diagram stretched anisotropically only along the $OX$ axis.

Note that, as Jack characters are polynomial functions on Young diagrams,
they can be defined on non-integer dilatation or anisotropical stretching of Young diagrams
(in fact, they can be defined on any generalized Young diagrams,
see \cite{DolegaF'eray'Sniady2008} for details).
In the case of zonal characters, this corresponds to writing
Theorem~\ref{theo:ZonalFeraySniady-B} for sequences $\pp$ and $\qq$ with
non-integer terms.

Following Biane \cite{Biane1998} (who used a different, but equivalent
definition), for a Young diagram $\lambda$ we define its \emph{free cumulants}
$R_2(\lambda),R_3(\lambda),\dots$ by the formula
$$R_{k}(\lambda) = \lim_{s\to\infty} \frac{1}{s^{k}}
\Sigma_{k-1}^{(1)}(D_s\lambda).$$
In other words, each free cumulant $R_{k}(\lambda)$ is asymptotically the
dominant term of the character on a cycle of length $k-1$ in the limit when the
Young diagram tends to infinity. It is natural to generalize this definition
using Jack characters:
$$R_{k}^{(\alpha)}(\lambda) = \lim_{s\to\infty} \frac{\alpha}{(\alpha s)^{k}}
\Sigma_{k-1}^{(\alpha)}(D_s\lambda).$$

In fact, the general $\alpha$ case can be expressed simply in terms of the
usual free cumulant thanks to \cite[Theorem 7]{Lassalle2009}:
$$R_{k}^{(\alpha)}(\lambda) = \frac{1}{\alpha^k} R_k(\alpha \lambda).$$
The quantities $R_{k}^{(\alpha)}(\lambda)$ are called $\alpha$-anisotropic free
cumulants of the Young diagram $\lambda$.

With this definition free cumulants might seem to be rather abstract quantities,
but in fact they could be equivalently defined in a very explicit way using
the shape of the diagram and linked to free probability, whence their name,
see \cite{Biane1998}. The equivalence of these two descriptions
makes them very useful parameters for describing Young diagrams.
Moreover, Proposition 2 and the Theorem of section 3 in \cite{Lassalle2008a}
imply that they form a homogeneous algebraic basis of the ring of shifted
symmetric functions. Therefore many interesting functions can be written in terms of
free cumulants. These features make free cumulants a perfect tool in the study
of asymptotic problems in representation theory, see for example
\cite{Biane1998,'Sniady2006c}.

\subsubsection{Kerov polynomials for Jack characters}
\label{subsect:classical_results_Kerov}
The following observation is due to Lassalle \cite{Lassalle2009}. Let $k\geq 1$
be a fixed integer and let $\alpha$ be fixed. Since
$\Sigma^{(\alpha)}_{k}$ is an $\alpha$-shifted
symmetric function and the anisotropic free cumulants
$(R_{l}^{(\alpha)})_{l\geq 2}$ form an algebraic basis of the ring of
$\alpha$-shifted symmetric functions, there exists a
polynomial $K^{(\alpha)}_k$ such that, for any Young diagram $\lambda$,
$$\Sigma^{(\alpha)}_{k}(\lambda)=
K^{(\alpha)}_k\big(R_2^{(\alpha)}(\lambda),R_3^{(\alpha)}(\lambda),
\ldots\big).$$
This polynomial is called \emph{Kerov polynomial for Jack character}.

Thus Kerov polynomials for Jack characters express Jack characters on cycles in terms of free
cumulants. For more complicated conjugacy classes it turns out to be more
convenient to express not directly the characters
$\Sigma^{(\alpha)}_{(k_1,\dots,k_\ell)}$ but rather \emph{cumulant}
\begin{displaymath}
 (-1)^{\ell-1} \kappa^{\text{id}}( \Sigma^{(\alpha)}_{k_1},
\dots, \Sigma^{(\alpha)}_{k_\ell} ).
\end{displaymath}
This gives rise to \emph{generalized Kerov polynomials for Jack characters}, denoted
$K^{(\alpha)}_{(k_1,\dots,k_\ell)}$. In the classical context $\alpha=1$ these
quantities have been introduced by one of us and Rattan
\cite{Rattan'Sniady2008}; in the Jack case they have been studied by Lassalle
\cite{Lassalle2009}. We skip the definitions and refer to the above papers for
details since generalized Kerov polynomials are not of central interest for
this paper.

      \subsubsection{Classical Kerov polynomials}
For $\alpha=1$ these polynomials are called simply \emph{Kerov polynomials}.
This case has a much longer history and it was initiated by Kerov
\cite{Kerov2000talk} and Biane \cite{Biane2003} who proved that in this
case the coefficients are in fact integers and conjectured their non-negativity.
This conjecture has been proved by the first-named author \cite{F'eray2009},
also for generalized Kerov polynomials.
Then, an explicit combinatorial interpretation has been
given by the authors, together with Do{\l}\k{e}ga, in
\cite{DolegaF'eray'Sniady2008}, using a different method.

These polynomials have a
deep structure, from a combinatorial and analytic point of view, and there are
still open problems concerning them. For a quite comprehensive bibliography on
this subject we refer to \cite{DolegaF'eray'Sniady2008}.

Most of properties of Kerov polynomials seem to be generalizable in the case of
a general value of the parameter $\alpha$, although not much has been
proved for the moment (see \cite{Lassalle2009}).

      \subsection{The main result 3: Kerov's polynomials for zonal characters}

As in the classical setting, the coefficients of zonal Kerov polynomials
have a nice combinatorial interpretation,
analogous to the one from \cite{DolegaF'eray'Sniady2008}.
Namely, if we denote $\big[x_1^{v_1} \cdots x_t^{v_t}\big] P$ the coefficient of 
$x_1^{v_1} \cdots x_t^{v_t}$ in $P$, we show the following result.

\begin{theorem}
\label{theo:kerov}
    Let $\mu$ be a partition of the integer $k$ and $(S_1, S_2)$ be a fixed
    couple of pair-partitions of\/ $[2k]$ of type $\mu$.
    Let $s_2,s_3,\dots$ be a sequence of non-negative
    integers with only finitely many non-zero elements.

Then the rescaled coefficient
$$  (-1)^{|\mu|+\ell(\mu)+2s_2+3s_3+\cdots}\ 2^{\ell(\mu) - (2s_2+3s_3+\cdots)}
\left[ \left(R_2^{(2)}\right)^{s_2}
\left(R_3^{(2)}\right)^{s_3} \cdots\right]
K^{(2)}_\mu $$ of the (generalized)
zonal Kerov polynomial is equal to the number of pairs
$(S_0,q)$ with the following properties:
\begin{enumerate}[label=(\alph*)]
 \item \label{enum:first-condition}
$S_0$ is a pair-partition of\/ $[2k]$ such that the three involutions
corres\-ponding to
$S_0$, $S_1$ and $S_2$ generate a transitive subgroup of $\Sym{2k}$;
 \item \label{enum:ilosc2} the number of loops in $\loops(S_0,S_1)$ is
equal to $s_2+s_3+\cdots$;
 \item \label{enum:boys-and-girls} the number of loops in $\loops(S_0,S_2)$ is equal
to $ s_2+2 s_3+3 s_4+\cdots$;
 \item \label{enum:kolorowanie} $q$ is a function from the set $\loops(S_0,S_1)$ 
 to the set $\{2,3,\dots\}$; we require that each
number
$i\in\{2,3,\dots\}$ is used exactly $s_i$ times;
\item \label{enum:marriage}
for every subset $A \subset \loops(S_0,S_1)$ 
which is nontrivial ({\it i.e.}, $A\neq\emptyset$ and $A \neq \loops(S_0,S_1)$), there are
more than $\sum_{v \in A} \big( q(v)-1 \big)$ loops in $\loops(S_0,S_2)$ which have
a non-empty intersection with at least one loop from $A$.
\end{enumerate}
\end{theorem}
Condition \ref{enum:marriage} can be reformulated in a number of equivalent
ways \cite{DolegaF'eray'Sniady2008}.
This result will be proved in Section \ref{sec:kerov}.

\begin{example}
    We continue the previous example: $\mu=(2)$,
    $S_1=\big\{ \{1,2\},\{3,4\} \big\}$ and 
    $S_2=\big\{ \{1,4\},\{2,3\} \big\}$.
    Recall that $S_0$ can take three values ($S_1$, $S_2$ and another
    value $S_3= \big\{ \{1,3\},\{2,4\} \big\}$).
    In each case, condition \ref{enum:first-condition} is fulfilled.
    The number of loops in $\loops(S_0,S_1)$ and $\loops(S_0,S_2)$
    in each case was already calculated in the example on page \pageref{ex:example};
    from the discussion there it follows as well that
    any $\ell \in\loops(S_0,S_1)$ and any $\ell' \in \loops(S_0,S_2)$
    have a non-empty intersection.

    \begin{itemize}
        \item If $S_0=S_2$ (resp.~$S_0=S_3$), conditions \ref{enum:ilosc2}, \ref{enum:boys-and-girls},
            \ref{enum:kolorowanie} and \ref{enum:marriage} are fulfilled for
            $(s_2,s_3,\dots)=(0,1,0,0,\dots)$ (respectively, $(s_2,s_3,\dots)=(1,0,0,\dots)$)
            and $q$ associating $3$ (resp.~$2$) to the unique loop of $\loops(S_0,S_1)$.
        \item If $S_0=S_1$, conditions \ref{enum:ilosc2} and \ref{enum:boys-and-girls}
            cannot be fulfilled at the same time for any sequence $(s_i)$ because
            this would imply \[2 = |\loops(S_0,S_1)| \leq |\loops(S_0,S_2)| = 1.\]        
    \end{itemize}
    Finally, all coefficients of $K_{(2)}^{(2)}$ are equal to $0$, except for:
    \begin{align*}
        \frac{-1}{2}\ [R_2^{(2)}] K_{(2)}^{(2)} &= 1; \\
        \frac{1}{4}\ [R_3^{(2)}] K_{(2)}^{(2)} &= 1.
    \end{align*}
    In other terms,
    \[ K_{(2)}^{(2)} = 4 R_3^{(2)} - 2 R_2^{(2)}. \]
    This fits with Lassalle's data \cite[top of page 2230]{Lassalle2009}.
\end{example}

   \subsection{Symplectic zonal polynomials}
As mentioned above, the case $\alpha=\frac{1}{2}$ is also special for Jack
polynomials, as we recover the so-called symplectic zonal polynomials. These
polynomials appear in a quaternionic analogue of James' theory, see
\cite[VII.6]{Macdonald1995}.

The symplectic zonal and zonal cases are linked by the duality
formula for Jack characters (see \cite[Chapter VI, equation
(10.30)]{Macdonald1995}):
\begin{equation}
    \label{eq:OneOverAlpha}
\theta_\rho^{(\alpha)}(\lambda) = (- \alpha)^{|\rho| - \ell(\rho)}
\ \theta_\rho^{(\alpha^{-1})}(\lambda'),
\end{equation}
where 
$\lambda'$ is conjugate of the partition $\lambda$. 


Using the definition of Jack characters, this equality becomes:
\begin{equation}
    \label{eq:OneOverAlphaInSigma}
\Sigma_\mu^{(\alpha)}(\lambda) = (- \alpha)^{|\mu| - \ell(\mu)}
\ \Sigma_\mu^{(\alpha^{-1})}(\lambda').
\end{equation}
Therefore the combinatorial interpretation of Stanley's and Kerov's polynomials for
zonal characters have analogues in the symplectic zonal case.
As it will be useful in the next section, let us state the one for Kerov's
polynomials.

\begin{theorem}
\label{theo:kerov-symplectic}
    Let $\mu$ be a partition of the integer $k$ and $(S_1, S_2)$ be a fixed
    couple of pair-partitions of\/ $[2k]$ of type $\mu$.
    Let $s_2,s_3,\dots$ be a sequence of non-negative
    integers with only finitely many non-zero elements.

Then the rescaled coefficient
$$  2^{|\mu|}
\left[ \left(R_2^{(1/2)}\right)^{s_2}
\left(R_3^{(1/2)}\right)^{s_3} \cdots\right]
K^{(1/2)}_\mu $$ of the (generalized) symplectic
zonal Kerov polynomial is equal to the number of pairs
$(S_0,q)$ with properties 
\ref{enum:first-condition}, \ref{enum:ilosc2}, \ref{enum:boys-and-girls},
\ref{enum:kolorowanie} and \ref{enum:marriage} of Theorem \ref{theo:kerov}.
\end{theorem}
\begin{proof}
This comes from Eq.~\eqref{eq:OneOverAlphaInSigma}, Theorem \ref{theo:kerov}
and the fact that:
\begin{multline*}
R_k^{(1/2)}(\lambda)=2^k R_k(1/2 \lambda) = 
2^k R_k \left[ D_{(1/2)} \big( ( 2 \lambda' )' \big) \right] \\
=R_k \big[  ( 2 \lambda' )' \big] = (-1)^k R_k( 2 \lambda' )=
(-2)^k R_k^{(2)}(\lambda').\qedhere
\end{multline*}
\end{proof}

      \subsection{Lassalle's conjectures}
In a series of two papers \cite{Lassalle2008a,Lassalle2009}
Lassalle proposed some conjectures on the expansion of Jack
characters in terms of Stanley's coordinates and free cumulants.
These conjectures suggest the existence of a combinatorial description of Jack
characters.
Our results give such a combinatorial description in the case of zonal characters.
Moreover, we can prove a few statements which are corollaries of Lassalle's conjectures.

Let us begin by recalling the latter (\cite[Conjecture 1]{Lassalle2008a} and
    \cite[Conjecture 2]{Lassalle2009}).
\begin{conjecture}\label{ConjectureLassalle}
Let $\mu$ be a partition of $k$.
    \begin{itemize}
        \item $(-1)^k \Sigma^{(\alpha)}_\mu(\pp,-\qq)$ is a polynomial in variables
            $\pp$, $\qq$ and $\alpha-1$ with
	    non-negative integer coefficients.
        \item there is a ``natural'' way to write the quantity 
 \[\kappa^{\text{id}}( \Sigma^{(\alpha)}_{k_1},\dots, \Sigma^{(\alpha)}_{k_\ell} )\]
 as a polynomial in the variables $R_i^{(\alpha)}$, $\alpha$
and $1-\alpha$ with
 non-negative integer coefficients.
    \end{itemize}
\end{conjecture}
In fact, Lassalle conjectured this in the case where $\mu$ has no part equal to
$1$, but it is quite easy to see that if it is true for some partition $\mu$, it
is also true for $\mu \cup 1$.

Having formulas only in the cases $\alpha=1/2$ and $\alpha=2$, we can not
prove this conjecture. In the following we will present a
few corollaries of Conjecture \ref{ConjectureLassalle} in the special
cases $\alpha=2$ and $\alpha=1/2$ and we shall prove them.
This gives an indirect evidence supporting Conjecture \ref{ConjectureLassalle}.

\begin{proposition}\label{PropStanleyLassalle2}
    Let $\mu$ be a partition of $k$. Then $(-1)^k \Sigma^{(2)}_\mu(\pp,-\qq)$
    is a polynomial in variables $\pp$, $\qq$ with non-negative integer coefficients.
\end{proposition}
If we look at the expansion of symplectic zonal polynomials in Stanley's
coordinates, Lassalle's conjecture does not imply neither integrity nor
positivity of the coefficients as we specialize the variable $\alpha-1$ to a
non-integer negative value.

\begin{proposition}\label{PropKerovLassalle2}
    Let $\mu$ be a partition of $k$. Then $K_\mu^{(2)}$ has
integer coefficients.
\end{proposition}
In this case there is no positivity result, because one of the variables of the
polynomial, namely $1-\alpha$, is specialized to a negative value.

\begin{proposition}\label{PropKerovLassalleHalf}                    
    Let $\mu$ be a partition of $k$. Then $K_\mu^{(1/2)}$ has
non-negative rational coefficients.
\end{proposition}
\begin{proof}
It is a direct consequence of Theorem \ref{theo:kerov-symplectic}.
\end{proof}
In this case there is no integrality result, because the
variables $\alpha$ and $1-\alpha$ are specialized to non-integer values.

Propositions \ref{PropStanleyLassalle2} and \ref{PropKerovLassalle2} 
 are proved in Sections
\ref{SubsectStanleyLassalle2} and \ref{SubsectKerovLassalle}.

%
%
      \subsection{Pair-partitions and zonal characters: the dual picture}
It should be stressed that there was another result linking triplets of
pair-partitions and zonal characters; it can be found in the work of Goulden
and Jackson \cite{Goulden1996a}. But their result goes in the reverse direction
than ours: they count triplets of pair partitions with some properties using
zonal characters, while we express zonal characters using triplets of
pair-partitions. An analogous picture exists for pairs of permutations and the
usual characters of symmetric groups. It would be nice to understand the link
between these two dual approaches.

\subsection{Maps on possibly non-orientable surfaces}
Most of our theorems involve triplets of pair-partitions.
This combinatorial structure is in fact much more natural
than it might seem at first glance,
as they are in correspondence with graphs drawn on
(possibly non orientable and non connected) surfaces.
In section \ref{SectCombinatorics}, we explain this relation
and give combinatorial reformulations of our main results.

    \subsection{Overview of the paper}
Sections \ref{sec:zonal-polynomials}, \ref{sec:Stanley} and \ref{sec:kerov}
are respectively devoted to the proofs of the main results 1, 2 and 3.
Section \ref{SectCombinatorics} is devoted to the link with maps.

\section{Formulas for zonal polynomials}\label{sec:zonal-polynomials}

The main result of this section is Theorem \ref{theo:zonal-polynomials},
which gives a combinatorial formula for zonal polynomials. 

   \subsection{Preliminaries}
\label{subsec:definition}
In this paragraph we give the characterization of zonal polynomials, which is
the starting point of our proof of Theorem \ref{theo:zonal-polynomials}.
This characterization is due to James \cite{James1961}.
However, we will rather base our presentation on the section VII.3 of
Macdonald's book \cite{Macdonald1995}, because the link with more usual definitions
of zonal polynomials (as particular case of Jack symmetric functions, Eq.~(VII,
2.23) or {\it via zonal spherical functions} (VII, 2.13) is explicit there.

Consider the space $P(G)$ of polynomial functions on the set $G=\Gl_d(\R)$, 
{\it i.e.~}functions which are polynomial in the entries of the matrices. 
The group $G$ acts canonically on this space as follows: for $L,X\in G$ and $f
\in P(G)$, we define
$$ (L f)(X) = f(L^T X). $$
As a representation of $G$, the space $P(G)$ decomposes as $P(G) =
\bigoplus_\mu P_\mu$, where the sum runs over partitions of length at most
$d$ and where $P_\mu$ is a sum of representations of type $\mu$ \cite[Eq.~(VII,
3.2)]{Macdonald1995}.

Let us denote $K=O(d)$. We will look particularly at the subspace $P(G,K)$ of
functions $f\in P(G)$ which are left- and right-invariant under the action of
the orthogonal group, that is
such that, for any $k,k' \in K$ and $g \in G$,
\[ f(kgk')=f(g). \]
The intersection $P_\mu \cap P(G,K)$ has dimension $1$ if $\mu=2\lambda$ for some
partition $\lambda$ and $0$ otherwise \cite[Eq.~(VII,
3.15)]{Macdonald1995}.
Thus there is a unique function $\Omega^{(d)}_\lambda$ such that:
\begin{enumerate}[label=(\alph*)]
 \item 
$\Omega^{(d)}_\lambda(1_G) = 1$,
 \item
$\Omega^{(d)}_\lambda$ is invariant under the left action of
the orthogonal group $O_d(\R)$,
 \item
$\Omega^{(d)}_\lambda$ belongs to $P_{2\lambda}$.
\end{enumerate}
This function $\Omega^{(d)}_\lambda$ is linked to zonal polynomials by the
following
equation \cite[Eq.~(3.24)]{Macdonald1995}:
\[\Omega^{(d)}_\lambda(X) = \frac{Z_\lambda ( \Sp(X X^T) )}{Z_\lambda(1^d)},\]
where $\Sp(X X^T)$ is the multiset of eigenvalues of $X X^T$.
Therefore if we find functions $\Omega^{(d)}_\lambda$ with the properties above,
we will be able to compute zonal polynomials up to a multiplicative constant.

We will look for such functions in a specific form.
For $Z=v_1\otimes \cdots\otimes v_{2n}\in (\R^d)^{\otimes 2n}$ we define a
homogeneous polynomial function of degree $2n$ on $G$
$$ \phi_Z(X) = \langle X^T v_1, X^T v_2 \rangle \cdots \langle X^T v_{2n-1}, X^T
v_{2n} \rangle \quad \text{for }X\in \M_d(\R) $$
and for general tensors $Z\in (\R^d)^{\otimes 2n}$ by linearity. Clearly, 
$$ \phi_Z( X O ) = \phi_Z( X) \qquad \text{for any } O \in O_d(\R);$$
in other words $\phi_Z$ is invariant under the right action of the orthogonal
group $O_d(\R)$.

Besides, $\Gl_d(\R)$ acts on $(\R^d)^{\otimes 2n}$:
this action is defined on elementary tensors by 
\begin{equation}
\label{eq:action}
 L(v_1\otimes \cdots\otimes v_{2n}) = L v_1 \otimes \cdots
\otimes 
 L v_{2n}.  
\end{equation}
%
\begin{lemma}
    The linear map $\phi:(\R^d)^{\otimes 2n} \to P(G)$ is an
    intertwiner of $G$-representation, {\it i.e.~}for all $g \in G$ and 
    $Z \in (\R^d)^{\otimes 2n} $ one has:
    \[ g \phi_Z = \phi_{g Z}. \vspace{-.7cm}\]   
    \label{LemIntertwiner}
\end{lemma}
\begin{proof}
    Straightforward from the definition of the actions.
\end{proof}
Thanks to this lemma, $\phi_{z^{(d)}_\lambda}$ will be left-invariant by multiplication
by the orthogonal group if and only if $z^{(d)}_\lambda$ is invariant by the action
of the orthogonal group. 
Besides, $\phi_{z^{(d)}_\lambda}$ is in $P_\mu$ if $z^{(d)}_\lambda$ itself in the
isotypic component of type $\mu$ in the representation $(\R^d)^{\otimes 2n}$.

Finally, we are looking for an element $z^{(d)}_\lambda \in (\R^d)^{\otimes 2n}$
such that:
\begin{enumerate}[label=(\alph*)]
 \item \label{item:non-zero}
     $\phi_{z^{(d)}_\lambda}$ is non-zero,
 \item \label{item:left-invariance}
$z^{(d)}_\lambda$ is invariant under the left action of $O_d(\R)\subset\Gl_d(\R)$,
 \item \label{item:representation}
     $z^{(d)}_\lambda$ belongs to the isotypic component of type $2\lambda$ in the
     representation $(\R^d)^{\otimes 2n}$ (in particular $n$ has to be the size of $\lambda$).
\end{enumerate}

In the following paragraphs we exhibit an element
$z^{(d)}_\lambda\in (\R^d)^{\otimes 2n}$ with these properties
and use it to compute the zonal polynomial $Z_\lambda$. 

\subsection{A few lemmas on pair-partitions}
\begin{lemma}
    Let $(S_1, S_2)$ be a couple of pair-partitions of $[2n]$ of type $\mu$.
    Then if we see $S_1$ and $S_2$ as involutions of $[2n]$,
    their composition $S_1 \circ S_2$ has cycle-type $\mu \cup \mu$.
\end{lemma}
\begin{proof}
    Let $(i_1,i_2,\ldots,i_{2\ell})$ be a loop of length $2\ell$ in the graph
    $\loops(S_1,S_2)$. This means that, up to a relabeling, $S_1$ (resp.~$S_2$)
    contains the pairs
    $\{i_{2j}, i_{2j+1}\}$ (resp.~$\{i_{2j-1},i_{2j}$) for $1 \leq j \leq \ell$
    (with the convention $i_{2\ell+1}=i_1$). Then the restriction of 
    $S_1 \circ S_2$ to $\{i_1,\cdots,i_{2\ell}\}$
    \[(S_1 \circ S_2) \big|_{ \{i_1,\cdots,i_{2\ell}\} }
    = (i_1\ i_3\ \cdots\ i_{2\ell} -1)(i_2\ i_4\ \cdots\ i_{2\ell}) \]
    is a disjoint product of two cycles of length $\ell$. 
    The same is true for the restriction to the support of each loop,
    therefore $S_1 \circ S_2$ has cycle-type
    $\mu_1,\mu_1,\mu_2,\mu_2,\ldots$
\end{proof}

The symmetric group $\Sym{2n}$ acts on the set of pair-partitions of $[2n]$:
if $\sigma$ is a permutation in $\Sym{2n}$ and $T$ a pair-partition of $[2n]$,
we denote by $\sigma \cdot T$ the pair partition such that
$\{\sigma(i),\sigma(j)\}$ is a part of  $\sigma \cdot T$
if and only if $\{i,j\}$ is a part of $T$.

\begin{lemma}
\label{lem:sign}
Let $\sigma$ be a permutation of the boxes of\/ $2\lambda$
which preserves each column. Then
$$ (-1)^{\sigma} = (-1)^{\loops(\sigma\cdot S,  S)}. $$
\end{lemma}
\begin{proof}
Young diagram $2 \lambda$ can be viewed as a concatenation of rectangular Young
diagrams of size $i\times 2$ ($i$ parts, all of them equal to $2$); for
this reason it is enough to proof the lemma for the case when
$2\lambda=i\times 2$. 
Permutation $\sigma$ can be viewed as a pair $(\sigma^{(1)}, \sigma^{(2)})$
where $\sigma^{(j)}\in\Sym{i}$ is the permutation of $j$-th column. Then
$$ (-1)^{\sigma}= (-1)^{\sigma^{(1)}} (-1)^{\sigma^{(2)}} = 
(-1)^{\sigma^{(1)} \left(  \sigma^{(2)} \right)^{-1} } =
(-1)^{(\ell_1-1)+(\ell_2-1)+\cdots}, $$
where $\ell_1,\ell_2,\dots$ are the lengths of the cycles of the permutation
$\sigma^{(1)} \left(\sigma^{(2)} \right)^{-1} $.

Let $(\Box[c,r])$ denote the box of the Young diagram in the column $c$ and
the row $r$. Then
\begin{multline*}
\sigma S \sigma^{-1} S (\Box[1,i]) = \sigma S \sigma^{-1}(\Box[2,i]) = \sigma S
\big(\Box[2,(\sigma^{(2)})^{-1}(i)]\big) \\
 = \sigma\big(\Box[1,(\sigma^{(2)})^{-1}(i)]\big) = \Box 
[1,\sigma^{(1)} \left(\sigma^{(2)} \right)^{-1}(i) ].
\end{multline*}
So $\sigma S \sigma^{-1} S = (\sigma \cdot S) S$ permutes the first column and
its restriction to the first column has cycles of length $\ell_1,\ell_2,\dots$.
The same is true for the second column. It follows that $(\sigma \cdot S) S$ has
cycles of length $\ell_1,\ell_1,\ell_2,\ell_2,\dots$ or, equivalently, the
lengths of the loops of $\loops(\sigma \cdot S, S) $ are equal to
$2\ell_1,2\ell_2,\dots$ which finishes the proof.
\end{proof}

The last lemma of this paragraph concerns the structure of the set of couples
of pair-partitions of $[2n]$ endowed with the diagonal action of the symmetric
group. From the definition of the graph $\loops(S_1,S_2)$ it is clear that 
$\loops(\sigma S_1,\sigma S_2)$ and $\loops(S_1,S_2)$ are isomorphic as
bipartite graphs, thus they have the same type. Conversely:

\begin{lemma}\label{LemCountCouples}
    The set of couples $(S_1,S_2)$ of type $\mu$ forms exactly
    one orbit under the diagonal action of the symmetric group $\Sym{2n}$.
    Moreover, there are exactly $\frac{(2n)!}{z_\nu 2^{\ell(\nu)}}$ of them.
\end{lemma}
\begin{proof}
    Let us consider two couples $(S_1,S_2)$ and $(S_1',S_2')$ of type $\mu$
    such that both graphs $G:=\loops(S_1,S_2)$ and $G':=\loops(S_1',S_2')$ are
    collections of loops of lengths $2\mu_1,2\mu_2\dots$.
    These two graphs are isomorphic as vertex-bicolored graphs.
    Let $\varphi$ be any isomorphism of them.
    As it sends the edges of $G$ to the edges of $G'$, it can be
    seen as a permutation in $\Sym{2n}$.
    As it sends the black (resp.~white) vertices of $G$ to the black
    (resp.~white) vertices of $G'$,
    one has: $\varphi(S_1)=S'_1$ (resp.~$\varphi(S_2)=S'_2$).
    Thus all couples of pair-partitions of type $\mu$ are in the same orbit.

    Fix a couple $(S_1,S_2)$ of type $\mu$ and denote by
    $L_1,\dots,L_{\ell(\mu)}$ the loops of the graph $\loops(S_1,S_2)$.
    Moreover we fix arbitrarily one edge $e_i$ in each loop $L_i$.
    Let $\sigma$ belong to the stabilizer of the action of $\Sym{2n}$ on a
    $(S_1,S_2)$; in other words $\sigma$ commutes with $S_1$ and $S_2$.    
    Such a $\sigma$ induces a permutation $\tau$ of the loops $(L_i)$ respecting
    their sizes; there are $\prod_i m_i(\mu)!$ such permutations.
    Besides, once $\tau$ is fixed, there are $2\mu_i$ possible images for $e_i$
    (it can be any element of the loop $\tau(L_i)$, which has the same size
    as $L_i$ which is equal to  $2\mu_i$).
    As $\sigma(S_j(i))=S_j(\sigma(i))$ for $j=1,2$, the permutation $\sigma$
    is entirely determined by the values of $\sigma(e_i)$.
    Conversely, if we fix $\tau$ and some compatible values for $\sigma(e_i)$,
    there is one permutation $\sigma$ in the centralizer of $S_1$ and $S_2$
    corresponding to these values.
    Finally, the cardinality of this centralizer
    is equal to $z_\mu\ 2^{\ell(\mu)}=\prod_i m_i(\mu)!\ (2i)^{m_i(\mu)}$.
\end{proof}

\subsection{Pair-partitions and tensors}\label{SubsectPairPartitionsTensors}
%
%
%

If $P$ is a pair-partition of the ground set $[2n]$, we will associate
to it the tensor
$$ \Psi_P = \sum_{1\leq i_1,\dots,i_{2n} \leq d} \delta_{P}(i_1,\dots,i_{2n})\
e_{i_1} \otimes \cdots \otimes e_{i_{2n}} \in (\R^d)^{\otimes 2n}, $$
where $\delta_{P}(i_1,\dots,i_{2n})$ is equal to $1$ if $i_k=i_l$ for all
$\{k,l\}\in P$ and is equal to zero otherwise. 
The symmetric group $\Sym{2n}$ acts on the set of pair-partitions
and on the set of tensors $(\R^{d})^{\otimes 2n}$ and
it is straightforward that
$P\mapsto \Psi_P$ is an intertwiner with respect to these two actions.

\begin{lemma}
Let $Z \in (\R^d)^{\otimes 2n}$. Then
$$\phi_Z(X)= \langle Z, X^{\otimes 2n} \Psi_S \rangle $$
with respect to the standard scalar product in $(\R^d)^{\otimes 2n}$, where
$S$, given by \eqref{eq:first-pair-partition}, is the first pair-partition.
\end{lemma}

\begin{proof}
We can assume by linearity that $Z=v_1 \otimes \ldots \otimes v_{2n}$. The right-hand side becomes:
\begin{multline*}
 \langle Z, X^{\otimes 2n} \Psi_S \rangle  = \\ \sum_{i_1,\dots,i_n}
\big\langle v_1 \otimes \cdots \otimes v_{2n},
  X e_{i_1}\otimes X e_{i_1} \otimes \cdots \otimes X e_{i_n}
\otimes X e_{i_n} \big\rangle \\
= \sum_{1\leq i_1,\dots,i_n \leq d} \ \ \prod_{j=1}^n \ \ \langle v_{2j-1}, X
e_{i_j} \rangle \cdot \langle v_{2j}, X e_{i_j} \rangle \\
= \prod_{j=1}^n \left[\sum_{1 \leq i \leq d} \langle X^T v_{2j-1}, e_i \rangle
\cdot \langle X^T v_{2j}, e_i \rangle \right]
= \prod_{j=1}^n \langle X^T v_{2j-1}, X^T v_{2j} \rangle. \qedhere
\end{multline*}
\end{proof}

\begin{lemma}\label{LemPhiPsiPowerSum}
Let $P$ be a pair-partition of $[2n]$ and $S$, as before, the
pair-partition of the same set given by \eqref{eq:first-pair-partition}. Then
\begin{multline*} \phi_{\Psi_P}(X)=\langle \Psi_P, X^{\otimes 2n}\Psi_S \rangle
= \Tr
\left[(XX^T)^{\ell_1}\right] 
\ \Tr \left[(XX^T)^{\ell_2}\right] \cdots \\ = p_{\loops(P,S)}(\Sp(X X^T)),
\end{multline*}
where $2\ell_1,2\ell_2,\dots$ are the lengths of the loops of $\loops(P,S)$.
\end{lemma}

\begin{proof}
Let us consider the case where $\loops(P,S)$ has only one loop of length $2
\ell$.
Define $P'=\big\{ \{2,3\},\{4,5\},\dots,\{2\ell-2,2\ell-1\},\{2\ell,1\}\big\}$
Then the couples $(P,S)$ and $(P',S)$ have the same type and thus,
by Lemma \ref{LemCountCouples}, there exists a permutation $\sigma$ such that
$\sigma \cdot P'= P$ and $\sigma \cdot S=S$. Then
\begin{multline*} 
    \langle \Psi_P, X^{\otimes 2n}\Psi_S \rangle
 = \langle \sigma \Psi_{P'}, X^{\otimes 2n} \sigma \Psi_S \rangle
 = \langle \sigma \Psi_{P'}, \sigma X^{\otimes 2n} \Psi_S \rangle \\
 = \langle \Psi_{P'}, X^{\otimes 2n}\Psi_S \rangle.
 \end{multline*}
 We used the facts that $P \mapsto \Psi_P$ is an intertwiner for the symmetric
 group action, that this action commutes with $X^{\otimes 2n}$ and
 that it is a unitary action.
 Therefore, it is enough to consider the case $P=P'$. In this case,
$$\Psi_P=\sum_{1 \leq j_1,\dots,j_\ell \leq d} e_{j_\ell} \otimes e_{j_1}
\otimes
e_{j_1} \otimes \dots \otimes e_{j_{\ell-1}} \otimes e_{j_{\ell-1}} \otimes
e_{j_\ell}.$$
Therefore one has:
\begin{align*}
 \phi_{\Psi_P}(X) &= \sum_{1 \leq j_1,\dots,j_\ell \leq d} \langle X^T
e_{j_\ell},
X^T e_{j_1} \rangle \cdot \langle X^T e_{j_1}, X^T e_{j_2} \rangle \cdots
\langle X^T e_{j_{\ell-1}}, X^T e_{j_\ell} \rangle \\
&= \sum_{1 \leq j_1,\dots,j_\ell \leq d} \langle X X^T e_{j_\ell}, e_{j_1}
\rangle
\cdot \langle X X^T e_{j_1}, e_{j_2} \rangle \cdots \langle X X^T
e_{j_{\ell-1}},
e_{j_\ell} \rangle\\
&= \sum_{1 \leq j_1,\dots,j_\ell \leq d} (X X^T)_{j_1,j_\ell} \cdot (X
X^T)_{j_2,j_1} \cdots (X X^T)_{j_\ell,j_{\ell-1}} \\ 
& = \Tr(X X^T)^\ell.
\end{align*}
The general case is simply obtained by multiplication of the above one-loop
case.
\end{proof}

It follows that $X\mapsto \phi_{\Psi_P}(X)$ is invariant under the left action of the orthogonal group
$O_d(\R)$.
The above discussion shows that if $P$ is a pair-partition (or, more generally,
a formal linear combination of pair-partitions) then condition
\ref{item:left-invariance} is fulfilled for $z_\lambda=\Psi_P$. For this reason
we will look for candidates for $z_\lambda$ corresponding to zonal polynomials
in this particular form.

   \subsection{Young symmetriser}


Let a partition $\lambda$ be fixed; we denote  $n=|\lambda|$.
We consider the Young tableau $T$ of shape $2\lambda$ in which boxes are
numbered consecutively along the rows. This tableau was chosen in such a way
that if we interpret the pair-partition $S$ as a pairing of the appropriate
boxes of $T$ then a box in the column $2i-1$ is paired with the box in the
column $2i$ in the same row, where $i$ is a positive integer (these two boxes
will be called neighbors in the Young diagram $2\lambda$).

Tableau $T$ allows us to identify boxes
of the Young diagram $2\lambda$ with the elements of the set $[2n]$.
In particular, permutations from $\Sym{2n}$ can be interpreted as permutations
of the boxes of $2\lambda$. We denote 
\begin{align*} 
P_{2\lambda}=& \{\sigma\in \Sym{2n} \colon \sigma \text{ preserves each row
of }2\lambda \}, \\
Q_{2\lambda}=& \{\sigma\in \Sym{2n} \colon \sigma \text{ preserves each column
of }2\lambda \}
\end{align*}
and define
\begin{align*} 
a_{2\lambda} = & \sum_{\sigma\in P_{2\lambda}} \sigma \in \C[\Sym{2n}],\\
b_{2\lambda} = & \sum_{\sigma\in Q_{2\lambda}} (-1)^{|\sigma|} \sigma \in
\C[\Sym{2n}],\\
c_{2\lambda} = &  b_{2\lambda} a_{2\lambda}.
\end{align*} 

The element $c_{2\lambda}$ is called \emph{Young symmetriser}.
There exists some non-zero scalar $\alpha_{2\lambda}$ such that
$\alpha_{2\lambda} c_{2\lambda}$ is a projection. Its image $\C[\Sym{2n}]
\alpha_{2\lambda} c_{2\lambda}$ under multiplication from the
right on the left-regular representation gives an irreducible representation
$\rho^{2\lambda}$ of the symmetric group (where the symmetric group acts by left
multiplication) associated to the Young diagram $2\lambda$
(see \cite[Theorem 4.3, p. 46]{FultonHarrisRepresentation}). 

Recall (see \cite[Corollary 1.3.14]{RepresentationsOVApproach})
that there is also a central projection in $\C[\Sym{2n}]$, denoted $\p_{2\lambda}$,
whose image $\C[\Sym{2n}] \p_{2\lambda}$ under multiplication from the
right (or, equivalently, from the left) on the left-regular representation is
the sum of all irreducible representations of type $\rho^{2\lambda}$
contributing to $\C[\Sym{2n}]$. It follows that
$\C[\Sym{2n}] c_{2\lambda}$ is a subspace of $\C[\Sym{2n}] \p_{2\lambda}$.
It follows that there is an inequality
\begin{equation}
\label{eq:inequality}
\alpha_{2\lambda} c_{2\lambda} \leq \p_{2\lambda} 
\end{equation}
between projections in $\C[\Sym{2n}]$, i.e.
$$\alpha_{2\lambda} c_{2\lambda} \p_{2\lambda} =
\p_{2\lambda} \alpha_{2\lambda} c_{2\lambda}  =
\alpha_{2\lambda} c_{2\lambda}.$$

\subsection{Schur-Weyl duality}\label{SubsectSchurWeyl}
The symmetric group $\Sym{2n}$ acts on the vector space $(\R^d)^{\otimes 2n}$
by permuting the factors and the linear group $\Gl_d(\R)$ acts on the same space
by the diagonal action \eqref{eq:action}. 
These two actions commute and Schur-Weyl duality 
(see \cite[paragraph A.8]{Macdonald1995}) asserts that,
as a representation of $\Sym{2n} \times \Gl_d(\R)$, one has:
$$(\R^d)^{\otimes 2n} \simeq \bigoplus_{\mu \vdash 2n} V_\mu \times U_\mu,$$
where $V_\mu$ (resp.~$U_\mu$) is the irreducible representation of $\Sym{2n}$
(resp.~$\Gl_d(\R)$) indexed by $\mu$ (as we assumed in Section 
\ref{subsec:definition} that $d \geq 2n$, the representation $U_\mu$
does always exist).
But $\p_{2\lambda}(V_\mu)=\delta_{\mu,2\lambda} V_\mu$,
therefore the image $\p_{2\lambda} \left( (\R^d)^{\otimes 2n}\right)$ of the
projection
$\p_{2\lambda}$ is, as representation of $\Gl_d(\R)$, a sum of some number of
copies of the irreducible representation of $\Gl_d(\R)$ associated to the
highest weight $2\lambda$.
Using inequality \eqref{eq:inequality}, we know that
$\alpha_{2\lambda} c_{2\lambda} \left( (\R^d)^{\otimes 2n}\right)$
is a subspace of $\p_{2\lambda}
\left( (\R^d)^{\otimes 2n} \right)$. In this way, we proved that
$\alpha_{2\lambda} c_{2\lambda} \left( (\R^d)^{\otimes 2n}\right)$ is a
representation of $\Gl_d(\R)$ which is a sum of some number of copies of 
the irreducible representation of $\Gl_d(\R)$ associated with the highest weight
$2\lambda$.

Thus the element $c_{2\lambda} \cdot \Psi_S$ of $(\R^d)^{\otimes 2n}$
fulfills condition \ref{item:representation}.


   \subsection{A tensor satisfying James' conditions}
   Using the results of Section \ref{SubsectPairPartitionsTensors}
   and \ref{SubsectSchurWeyl}, we know that
   $$ z^{(d)}_\lambda:=\Psi_{c_{2\lambda} \cdot S} = 
c_{2\lambda} \Psi_S \in (\R^d)^{\otimes 2n} $$
fulfills conditions \ref{item:left-invariance} and \ref{item:representation}.

Therefore, as explained in Section \ref{subsec:definition},
if $\phi_{z^{(d)}_\lambda}$ is non-zero,
there exists a constant $C_\lambda$ such that:
$$\phi_{z^{(d)}_\lambda}(X) = C_\lambda Z_\lambda(\Sp(X X^T)).$$
Of course this is true also if the left hand-side is equal to zero.
Besides, using Lemma \ref{LemPhiPsiPowerSum}, one gets:
\begin{multline*} 
\phi_{c_{2\lambda} \Psi_S }(X)  = 
\sum_{\sigma_1\in Q_{2\lambda}}
\sum_{\sigma_2\in
P_{2\lambda}} (-1)^{\sigma_1}   
\langle  \Psi_{\sigma_1 \sigma_2 \cdot S}, 
X^{\otimes 2n} \Psi_S \rangle \\
 = 
\sum_{\sigma_1\in Q_{2\lambda}}
\sum_{\sigma_2\in
P_{2\lambda}} (-1)^{\sigma_1} \   
p_{\loops(\sigma_1 \sigma_2 \cdot S ,S)} (\Sp(X X^T)),
\end{multline*}
where the power-sum symmetric functions $p$ should be understood as in
\eqref{eq:power-sum}. 
Finally, we have shown that 
\begin{displaymath}
    Y_\lambda := \sum_{\sigma_1\in Q_{2\lambda}}    
    \sum_{\sigma_2\in P_{2\lambda}} (-1)^{\sigma_1} \
    p_{\loops(\sigma_1 \sigma_2 \cdot S, S)}
\end{displaymath}
and $C_\lambda Z_\lambda$ have the same evaluation on $\Sp(X X^T)$.
As this is true for all $X \in \Gl_d$ and all $d \geq 2|\lambda|$,
the two symmetric function $Y_\lambda$ and $C_\lambda Z_\lambda$ are equal.
We will use this fact in the following.

   \subsection{End of proof of Theorem \ref{theo:zonal-polynomials}}
      \label{subsec:proof-of-theo-zonal}

\begin{proof}
We know that
\begin{multline}
    C_\lambda Z_\lambda = \sum_{\sigma_1\in Q_{2\lambda}}    
        \sum_{\sigma_2\in P_{2\lambda}} (-1)^{\sigma_1} \
            p_{\loops(\sigma_1 \sigma_2 \cdot S, S)} \\
            = \sum_{\sigma_1\in Q_{2\lambda}}    
         \sum_{\sigma_2\in P_{2\lambda}} (-1)^{\sigma_1} \
             p_{\loops(\sigma_2 \cdot S, \sigma_1^{-1} \cdot S)}. 
             \label{EqPreuveMainTh}
\end{multline}
The set of pair-partitions which can be written as $\sigma_2 \cdot S$
with $\sigma_2 \in P_{2\lambda}$ is the set of pair-partitions of the boxes
of the Young diagram such that each pair
of connected boxes lies in the same row of the Young diagram
(we fixed the Young tableau $T$, so pair-partitions of the set $[2n]$
can be viewed as pair-partitions of the boxes of the Young diagram).
As $P_{2\lambda}$ is a group, each pair-partition in the orbit of $S$ can be
written as $\sigma_2 \cdot S$
with $\sigma_2 \in P_{2\lambda}$ in the same number of ways (say $C_2$).
Therefore, for any $\sigma_1\in Q_{2\lambda}$, 
\[\sum_{\sigma_2\in P_{2\lambda}} (-1)^{\sigma_1} \  
     p_{\loops(\sigma_2 \cdot S, \sigma^{-1} \cdot S)}    
     = C_2 \sum_{S_2} (-1)^{\sigma_1} \  
      p_{\loops(S_2, \sigma^{-1} \cdot S)},\]
where the sum runs over pair-partitions connecting boxes in the 
same row of $T$.

Analogously, the set of pair-partitions which can be written as $\sigma_1^{-1}
\cdot S$ for some $\sigma_1\in Q_{2\lambda}$ is the set of pair-partitions $S_1$
which match the elements of the $2j-1$ column of $T$ with the elements of the
$2j$-th column of $T$ for $1 \leq j \leq \lambda_1$ (it is equivalent to ask
that the boxes belonging to each cycle of $S_1 \circ S$ are in one column). As
before, such pair-partitions can all be written as $\sigma_1^{-1} \cdot S$ in
the same number of ways (say $C_1$). Besides, Lemma \ref{lem:sign} shows that
the sign $(-1)^{\sigma_1}$ depends only on 
$S_1 = \sigma_1^{-1} \cdot S$ and is equal to $(-1)^{\loops(S, S_1)}$.

Therefore, for any pair-partition $S_2$
\[\sum_{\sigma_1\in Q_{2\lambda}}
(-1)^{\sigma_1} \ p_{\loops(S_2, \sigma^{-1} \cdot S)}
= C_1 \sum_{S_1} (-1)^{\loops(S, S_1)} \ p_{\loops(S_2,S_1)}, \]
where the sum runs over pair-partitions $S_1$ such that
$S \circ S_1$ preserves each column of $T$.

Finally, Eq.~\eqref{EqPreuveMainTh} becomes
\begin{equation}
    \label{eq:zonal2}
    C_\lambda Z_\lambda= C_1 C_2 \sum_{S_1} \sum_{S_2} (-1)^{\loops(S,S_1)}
p_{\loops(S_1,S_2)},
\end{equation}
where the sum runs over $T$-admissible $(S_1,S_2)$. Recall that $T$-admissible
means that $S_2$ preserves each row of $T$ and $S \circ S_1$ preserves each
column.

To get rid of the numerical factors, we use the coefficient of $p_1^n$ in 
the power-sum expansion of zonal polynomials (given by Eq.~VI, (10.29)
in \cite{Macdonald1995}, see also Eqs.~VI, (10.27) and VII, (2.23)):
$$ [p_1^n] Z_\lambda = 1. $$
But the only pair of $T$-admissible pair-partitions $(S_1,S_2)$ such that
$\loops(S_1, S_2)$ is a union of $n$ loops (the latter implies automatically
that $S_1=S_2$) is $(S,S)$. Therefore the coefficient of $p_1^n$ in the double
sum of the right-hand side of \eqref{eq:zonal2} is equal to $1$ and finally:
   \[ Z_\lambda= \sum_{S_1} \sum_{S_2} (-1)^{\loops(S,S_1)}
p_{\loops(S_1,S_2)}. \qedhere \]
\end{proof}

\section{Formulas for zonal characters}\label{sec:Stanley}

This section is devoted to formulas for zonal characters; in particular,
the first goal is to prove Theorem \ref{theo:ZonalFeraySniady-A1}.

\subsection{Reformulation of Theorem \ref{theo:ZonalFeraySniady-A1}}
\label{subsec:reformulation}

Let $S_0$, $S_1$, $S_2$ be three pair-partitions of the set $[2k]$.
We consider the following function on the set of Young diagrams:
\begin{definition}
    Let $\lambda$ be a partition of any size. We define
    $N^{(2)}_{S_0,S_1,S_2}(\lambda)$ as the number of functions $f$ from
    $[2k]$ to the boxes of the Young diagram $2 \lambda$ such that
    for any $l\in[2k]$:
    \begin{enumerate}[label=(P\arabic*)]
        \addtocounter{enumi}{-1}
        \item
            \label{cond:P0}
            $f(l)$ and $f(S_0(l))$ are neighbors in the Young
            diagram $2\lambda$, \textit{i.e.}, if $f(l)$ is in the $2i+1$-th
            column (resp.~$2i+2$-th column), then $f(S_0(l))$ is the
	    box in the same row but
            in the $2i+2$-th column (resp.~$2i+1$-th column);
        \item
            \label{cond:P1}
            $f(l)$ and $f(S_0 \circ S_1(l))$ are in the same
            column;
        \item
            \label{cond:P2} $f(l)$ and $f(S_2(l))$ are in the same row.
    \end{enumerate}
    \label{def:Nb}

    We also define $\widehat{N}^{(2)}_{S_0,S_1,S_2}(\lambda)$ as the number of 
    injective functions fulfilling the above conditions.
\end{definition}

\begin{lemma}\label{lem:N2_N1}
    Let $S_0$, $S_1$, $S_2$ be pair-partitions.
    Then
    \[ N^{(2)}_{S_0,S_1,S_2} = 2^{|\loops(S_0,S_1)|} N^{(1)}_{S_0,S_1,S_2}. \]
\end{lemma}
\begin{proof}
Let $\lambda$ be a Young diagram
and let $f$ be a function $f : [2k] \to 2\lambda$ verifying properties
\ref{cond:P0}, \ref{cond:P1} and \ref{cond:P2}. We consider the projection $p :
2\lambda \to \lambda$, which consists of forgetting the separations between the
neighbors in $2\lambda$. More precisely, the boxes $(2i-1,j)$ and $(2i,j)$ of
$2\lambda$ are both sent to the box $(i,j)$ of $\lambda$. It is easy to check
that the composition $\overline{f}=p \circ f$ fulfills \ref{cond:Q0},
\ref{cond:Q1}, \ref{cond:Q2}.

Consider a function $g:[2k] \to \lambda$ verifying
\ref{cond:Q0}, \ref{cond:Q1} and \ref{cond:Q2}.
We want to determine functions $f$ verifying \ref{cond:P0}, \ref{cond:P1}
and \ref{cond:P2} such that $\overline{f}=g$.
If $g(k)$ (which is equal to $g(S_0(k))$ by condition \ref{cond:Q0}) is equal
to
a box $(i,j)$ of $\lambda$, then $f(k)$ and $f(S_0(k))$ belong to
$\{(2i-1,j),(2i,j)\}$.
Therefore, $f$ is determined by the parity of the column of $f(k)$ for each $k$.
Besides, if $f(k)$ is in an even-numbered (resp.~odd-numbered) column,
then $f(S_0(k))$ and $f(S_1(k))$ are in an odd-numbered (resp.~even numbered)
column (by conditions \ref{cond:P0} and \ref{cond:P1}).
Therefore, if we fix the parity of the column of $f(k)$ for some $k$, it is 
also fixed for $f(k')$, for all $k'$ in the same loop of $\loops(S_0,S_1)$.
Conversely, choose for one number $i$ in each loop of $\loops(S_0,S_1)$,
which of the two possible values should be assigned to $f(i)$.
Then there is exactly one function respecting these values and verifying
condition \ref{cond:P0}, \ref{cond:P1} and \ref{cond:P2}
(condition \ref{cond:P2} is fulfilled for each function $f$ such that
$\overline{f}$ verifies \ref{cond:Q2}).
Thus, to each function $g$ with properties \ref{cond:Q0}, \ref{cond:Q1}
and \ref{cond:Q2} correspond exactly $2^{|\loops(S_0,S_1)|}$ functions $f$
with properties \ref{cond:P0}, \ref{cond:P1} and \ref{cond:P2}.
\end{proof}


The above lemma shows that in order to show Theorem
\ref{theo:ZonalFeraySniady-A1} it is enough to prove the following equivalent
statement:
\begin{theorem}\label{theo:ZonalFeraySniady-A2}
    Let $\mu$ be a partition of the integer $k$ and $(S_1, S_2)$ be a fixed
    couple of pair-partitions of the set $[2k]$ of type $\mu$.
    Then one has the following equality between functions on the set of Young
    diagrams:
    \begin{displaymath}
        \Sigma^{(2)}_\mu = \frac{1}{2^{\ell(\mu)}}
            \sum_{S_0} (-1)^{\loops(S_0,S_1)} N^{(2)}_{S_0,S_1,S_2},
    \end{displaymath}
    where the sum runs over pair-partitions of $[2k]$.
\end{theorem}

We will prove it in Sections
\ref{subsec:extraction}--\ref{subsec:forgetting-injectivity}.

\subsection{Extraction of the coefficients}
      \label{subsec:extraction}
      Let $\mu$ and $\lambda$ be two partitions.
      In this paragraph we consider the case where $|\mu|=|\lambda|$.
If we look at the coefficients of a given power-sum function $p_\mu$ in
$Z_\lambda$,
using Theorem \ref{theo:zonal-polynomials}, one has:
\begin{displaymath}
    [p_\mu]Z_\lambda = \sum_{\substack{(S_1,S_2)\ T\text{-admissible}
\\ \type \loops(S_1,S_2)=\mu}}
    (-1)^{\loops(S , S_1)}.
\end{displaymath}
This equation has been proved in the case where $T$ and $S$ are, respectively,
the canonical Young tableaux and the first pair-partition, but the same proof
works for any
filling $T$ of $2\lambda$ by the elements of $[2|\lambda|]$ and
any pair-partition $S$ as long as $S$ matches the labels of the pairs of
neighbors of $2\lambda$ in $T$. As there
are $(2 |\lambda|)!$ fillings $T$ and one corresponding pair-partition $S=S(T)$
per
filling, one has:
\begin{displaymath}
    [p_\mu]Z_\lambda = \frac{1}{(2|\lambda|)!} \sum_{T}
    \sum_{\substack{(S_1,S_2)\ T\text{-admissible} \\
        \type\loops(S_1, S_2)=\mu}}
    (-1)^{\loops(S(T), S_1)},
\end{displaymath}
where the first sum runs over all bijective fillings of the diagram $2 \lambda$.
We can change the order of summation and obtain:
\begin{equation}
\label{EqPZSumS1S2}
[p_\mu]Z_\lambda =\frac{1}{(2|\lambda|)!}\!\!\!
\sum_{\substack{S_1,S_2\\ \type(S_1,S_2)=\mu}}\!\!\!
\left( \sum_{T}
(-1)^{\loops(S(T),S_1)}\ [(S_1,S_2)\text{ is $T$-admissible}]\right),
\end{equation}
where we use the convention that $[\text{condition}]$ is equal to $1$ if the
condition is true and is equal to zero otherwise.
Note that $\Sym{2n}$ acts on bijective fillings of $2\lambda$ by acting on
each box. It is straightforward to check that this action fulfills:
\begin{itemize}
    \item $S(\sigma \cdot T)=\sigma \cdot S(T)$;
    \item $(\sigma \cdot S_1,\sigma \cdot S_2)$ is $\sigma \cdot T$ admissible
        if and only if $(S_1,S_2)$ is $T$-admissible.
\end{itemize}

\begin{lemma}
The expression in the parenthesis in the 
right-hand side of Eq.~\eqref{EqPZSumS1S2} does not depend on $(S_1,S_2)$. 
\end{lemma}
\begin{proof}
Consider two couples $(S_1,S_2)$ and $(S'_1,S'_2)$, both of type $\mu$.
By Lemma \ref{LemCountCouples}, there exists a permutation $\sigma$ in
$\Sym{2n}$ such that $S'_1=\sigma \cdot S_1$ and $S'_2= \sigma \cdot S_2$.
Then 
\begin{multline*}
    \left( \sum_{T} (-1)^{\loops(S(T),S'_1)}\ 
    [(S'_1,S'_2)\text{ is $T$-admissible}]\right)\\
    = \left( \sum_{T} (-1)^{\loops(S(T),\sigma \cdot S_1)}\ 
    [(\sigma \cdot S_1,\sigma \cdot S_2)\text{ is $T$-admissible}]\right) \\
    = \left( \sum_{T} (-1)^{\loops(S(\sigma^{-1} \cdot T),S_1)}\
    [(S_1,S_2)\text{ is $\sigma^{-1} \cdot T$-admissible}]\right)\\
    = \left( \sum_{T'} (-1)^{\loops(S(T'),S_1)}\ 
        [(S_1,S_2)\text{ is $T'$-admissible}]\right),
\end{multline*}
where all sums run over bijective fillings of $2\lambda$. We used the fact that 
$T \mapsto \sigma \cdot T$ is a bijection of this set.
\end{proof}

Fix a couple of pair-partitions $(S_1,S_2)$ of type $\mu$.
As there are $\frac{(2|\mu|)!}{z_\mu 2^{\ell(\mu)}}$ couples of pair-partitions
of type $\mu$ (see Lemma \ref{LemCountCouples}), 
Eq.~\eqref{EqPZSumS1S2} becomes:
\[ [p_\mu] Z_\lambda = \frac{1}{z_\mu 2^{\ell(\mu)}}
   \left( \sum_{T'} (-1)^{\loops(S(T'),S_1)}\ 
       [(S_1,S_2)\text{ is $T'$-admissible}]\right).\]
As $|\mu|=|\lambda|$, one has:
\begin{multline*}
\Sigma^{(2)}_\mu(\lambda)= z_\mu\ [p_\mu]Z_\lambda = \frac{1}{2^{\ell(\mu)}}
\sum_{T} (-1)^{\loops(S(T),S_1)}\ [(S_1,S_2)\text{ is $T$-admissible}]\\
=\frac{1}{2^{\ell(\mu)}} \sum_{S_0} (-1)^{\loops(S_0,S_1)} 
\left( \sum_{T \text{ such that} \atop S(T)=S_0}
[(S_1,S_2)\text{ is $T$-admissible}] \right).
\end{multline*}

Bijective fillings $T$ of $2\lambda$ are exactly injective functions 
$f:[2n] \to 2\lambda$ (as the cardinality of two sets are the same,
such a function is automatically bijective).
Moreover, the conditions $S(T)=S_0$ and $(S_1,S_2)$ being $T$-admissible
correspond to conditions \ref{cond:P0}, \ref{cond:P1} and \ref{cond:P2}.
Using Definition \ref{def:Nb}, the last equality can be rewritten as follows:
when $|\mu|=|\lambda|$,
\begin{displaymath}
    \Sigma^{(2)}_\mu(\lambda)=\frac{1}{2^{\ell(\mu)}}
        \sum_{S_0} (-1)^{\loops(S_0,S_1)}
        \widehat{N}^{(2)}_{S_0,S_1,S_2}(\lambda).
    \label{eq:sigma_N_k=n}
\end{displaymath}

   \subsection{Extending the formula to any size}
Let us now look at the case where $|\mu|=k \leq n=|\lambda|$. We denote
$\widetilde{\mu} =\mu 1^{n-k}$. Then, using the formula above for
$z_{\widetilde{\mu}}\ [p_{\widetilde{\mu}}] Z_\lambda$, one has:
\begin{multline}
    \Sigma^{(2)}_\mu(\lambda) =  z_\mu\
	    \binom{n-k+m_1(\mu)}{m_1(\mu)} [p_{\widetilde{\mu}}]Z_\lambda
       = \frac{1}{(n-k)!} z_{\widetilde{\mu}} [p_{\widetilde{\mu}}]Z_\lambda
       \\
        = \frac{1}{2^{\ell(\mu)+ n-k}\ (n-k)!}
        \sum_{\widetilde{S_0}} (-1)^{\loops(\widetilde{S_0}, \widetilde{S_1})}
  \widehat{N}^{(2)}_{\widetilde{S_0},\widetilde{S_1},\widetilde{S_2}}(\lambda),
  \label{EqSigmaStilde}
\end{multline}
where $(\widetilde{S_1},\widetilde{S_2})$ is any fixed couple of pair-partitions
of type $\widetilde{\mu}$. We can choose it in the following way. Let
$(S_1,S_2)$ be a couple of
pair-partitions of the set $\{1,\dots,2k\}$ of type $\mu$ and define
$\widetilde{S_1}$
and $\widetilde{S_2}$ by, for $i=1,2$:
$$  \widetilde{S_i} = S_i \cup  
        \big\{ \{2k+1,2k+2\}, \dots, \{2n-1,2n\} \big\}. $$

\begin{lemma}\label{LemAnySize1}
With this choice of $(\widetilde{S_1},\widetilde{S_2})$, the quantity
$\widehat{N}^{(2)}_{\widetilde{S_0},\widetilde{S_1},\widetilde{S_2}}(\lambda)$
is equal to $0$ unless
\begin{equation}
\label{eq:structure-of-S0}
\widetilde{S_0}\big|_{\{2k+1,\dots,2n\}} =                          
         \big\{ \{2k+1,2k+2\}, \dots, \{2n-1,2n\} \big\}.
\end{equation}
\end{lemma}
\begin{proof}
    Let $\widetilde{S_0}$ be a pair-partition and $f[2n] \rightarrow 2\lambda$
    be a bijection verifying conditions \ref{cond:P0}, \ref{cond:P1}
    and \ref{cond:P2} with respect to the triplet
    $\widetilde{S_0},\widetilde{S_1},\widetilde{S_2}$.
    
    For any $l\geq k$, condition \ref{cond:P1} shows that
    $f(2l+1)$ and $f(\widetilde{S_0}(2l+2))$ are in the same column.
    In addition, condition \ref{cond:P0} shows
that $f(2l+2)$ and $f(\widetilde{S_0}(2l+2))$ are neighbors and hence are in the same row.
Besides, condition \ref{cond:P2} shows that $f(2l+1)$ and $f(2l+2)$ are in the
same row. In this way we proved that $f(2l+1)$ and $f(\widetilde{S_0}(2l+2))$ are in the
same row and column, hence $f(2l+1)=f(\widetilde{S_0}(2l+2))$.
As $f$ is one-to-one, one has $2l+1=\widetilde{S_0}(2l+2)$.
In this way we proved that the existence of an injective function $f$ satisfying
\ref{cond:P0}, \ref{cond:P1} and \ref{cond:P2} implies that
$2l+1=\widetilde{S_0}(2l+2)$ for all $l \geq k$.
\end{proof}

We need now to evaluate 
$\widehat{N}^{(2)}_{\widetilde{S_0},\widetilde{S_1},\widetilde{S_2}}(\lambda)$
when \eqref{eq:structure-of-S0} is fulfilled.
\begin{lemma}\label{LemAnySize2}
    Let us suppose that $\widetilde{S_0}$ fulfills
    Eq.~\eqref{eq:structure-of-S0}.
    Then denote $S_0=\widetilde{S_0}\big|_{\{1,\dots,2k\}}$. One has:
  \[\widehat{N}^{(2)}_{\widetilde{S_0},\widetilde{S_1},\widetilde{S_2}}(\lambda)
    =
  2^{n-k} (n-k)!\ \widehat{N}^{(2)}_{S_0,S_1,S_2}(\lambda).\]
\end{lemma}
\begin{proof}
    Let $\widetilde{f} : [2n] \rightarrow 2\lambda$ be a function counted in
 $\widehat{N}^{(2)}_{\widetilde{S_0},\widetilde{S_1},\widetilde{S_2}}(\lambda)$.
 Then it is straightforward to see that its restriction
 $\widetilde{f}\big|_{[2k]}$
 is counted in $\widehat{N}^{(2)}_{S_0,S_1,S_2}(\lambda)$.
 Conversely, in how many ways can we extent an injective function 
 $f : [2k] \hookrightarrow 2\lambda$ counted in 
 $\widehat{N}^{(2)}_{S_0,S_1,S_2}(\lambda)$ into a function
 $\widetilde{f} : [2n] \rightarrow 2\lambda$ counted in
 $\widehat{N}^{(2)}_{\widetilde{S_0},\widetilde{S_1},\widetilde{S_2}}(\lambda)$?
 One has to place the integers from $\{2k+1,\dots,2n\}$ in the $2(n-k)$ boxes
 of the set $2\lambda \setminus f([2k])$ such that numbers $2i-1$ and $2i$
 (for $k<i\leq n$) are in neighboring boxes.
 There are $2^{n-k}(n-k)!$ ways to place these number with this condition.
 If we obey this condition, then $\widetilde{f}$ verifies \ref{cond:P0},
     \ref{cond:P1} and \ref{cond:P2} with respect to 
     $(\widetilde{S_0},\widetilde{S_1},\widetilde{S_2})$.
 Therefore, any function $f$ counted in        
  $\widehat{N}^{(2)}_{S_0,S_1,S_2}(\lambda)$ is obtained as the restriction
  as exactly $2^{n-k}(n-k)!$ functions $\widetilde{f}$ counted in
 $\widehat{N}^{(2)}_{\widetilde{S_0},\widetilde{S_1},\widetilde{S_2}}(\lambda)$.
\end{proof}

With Eq.~\eqref{EqSigmaStilde}, Lemma \ref{LemAnySize1} and
Lemma \ref{LemAnySize2} it follows that the following equation holds true for
any partitions $\lambda$ and $\mu$ with $|\lambda| \geq |\mu|$ 
(notice also that it is also obviously true for $|\lambda| <  |\mu|$):
\begin{equation}
    \Sigma^{(2)}_\mu = \frac{1}{2^{\ell(\mu)}}
          \sum_{\substack{S_0 \text{ pair-partition} \\ \text{of
    } \{1,\dots,2|\mu|\} }} 
    (-1)^{\loops(S_0,S_1)} \widehat{N}^{(2)}_{S_0,S_1,S_2},
    \label{eq:sigma_N_inj}
\end{equation}
where $(S_1,S_2)$ is any couple of pair-partitions of type $\mu$.

   \subsection{Forgetting injectivity}
      \label{subsec:forgetting-injectivity}
In this section we will prove Theorem \ref{theo:ZonalFeraySniady-A2} (and
thus finish the proof of Theorem \ref{theo:ZonalFeraySniady-A1}).
In other terms, we prove that Eq.~\eqref{eq:sigma_N_inj} is still true if we
replace in each term of the sum $\widehat{N}^{(2)}_{S_0,S_1,S_2}$ by
$N^{(2)}_{S_0,S_1,S_2}$. In order to do this we have to check that, for any
\emph{non-injective} function $f : [2|\mu|] \to
2\lambda$, the total contribution
\begin{equation}
    \label{eq:contrib_f}
    \sum_{\substack{S_0 \text{ pair-partition} \\ \text{of }
    [2|\mu|] } }
    (-1)^{\loops(S_0,S_1)}\ [f\text{ fulfills \ref{cond:P0}, \ref{cond:P1} and
    \ref{cond:P2}}]
\end{equation}
of $f$ to the right-hand side of Eq.~\eqref{eq:sigma_N_inj} is equal to zero.

Let us fix a couple $(S_1,S_2)$ of pair-partitions of type $\mu$. 
We begin by a small lemma:
\begin{lemma}
    Let $f : [2k] \to 2\lambda$ be a function with $f(i)=f(j)$ for some $i$
    and $j$. Let us suppose that $f$ fulfills condition \ref{cond:P0} and
    \ref{cond:P1} with respect to some pair-partitions $S_0$ and $S_1$.
    Then, if $i$ and $j$ are the labels of edges in the same loop of
    $\loops(S_0,S_1)$ then there is an even distance between these two edges.
    \label{LemEvenDistance}
\end{lemma}
\begin{proof}
    If two edges labeled $k$ and $l$ are adjacent, this means that either
    $j=S_0(k)$ or $j=S_1(k)$.
    In both cases, as $f$ fulfills condition \ref{cond:P0} and \ref{cond:P1},
    the indices of the columns containing boxes $f(j)$ and $f(k)$ have
    different parities.
    Hence, the same is true if edges labeled $j$ and $k$ are in an odd distance
    from each other.
    As $f(i)=f(j)$, in particular they are in the same column and thus,
    the edges labeled $i$ and $j$ cannot be in the same loop with an odd
    distance between them.
\end{proof}

\begin{lemma}\label{LemChangingSignInvolution}
Let $f : [2|\mu|] \to 2\lambda$ with $f(i)=f(j)$. Then
\begin{enumerate}[label=\alph*)]
   \item \label{lemma-part-a}
	conditions \ref{cond:P0}, \ref{cond:P1} and \ref{cond:P2} are fulfilled
	for $S_0$ if and only if they are fulfilled for $S'_0 = (i\ j)\cdot
	S_0$;
   \item \label{lemma-part-b} 
	if these conditions are fulfilled, then
        \[(-1)^{\loops(S_0,S_1)} + (-1)^{\loops(S'_0 , S_1)}=0. \]
\end{enumerate}
\end{lemma}

\begin{proof}
    Recall that $S'_0$ is exactly the same pairing as $S_0$ except that
    $i$ and $j$ have been interchanged. Thus the part \ref{lemma-part-a} is
obvious
    from the definitions.

    Besides, the graph $\loops(S'_0,S_1)$ is obtained from $\loops(S_0,S_1)$
    by taking the edges with labels $i$ and $j$ and interchanging their black
    extremities. We consider two different cases.
    \begin{itemize}
     \item If $i$ and $j$ are in different loops $L_i$ and $L_j$ of the graph
         $\loops(S_0,S_1)$,
         then, when we erase the edges $i$ and $j$ we still have the same 
         connected components.
         To obtain $\loops(S'_0,S_1)$, one has to draw an edge between the 
         white extremity of $j$ and the black extremity of $i$.
         These two vertices were in different connected components $L_i$
         and $L_j$ of $\loops(S_0,S_1)$, therefore these two components are
         now connected and we have one less connected component.
         We also have to add another edge between the black extremity of $j$
         and the white extremity of $j$ but they are now in the same connected
         component so this last operation does not change the number
         of connected components.

         Finally, the graph $\loops(S'_0,S_1)$ has one less connected component
         than $\loops(S_0,S_1)$ and the part \ref{lemma-part-b} of the lemma
         is true in this case.

         This case is illustrated on Figure \ref{FigJoinLoops}.

  \begin{figure}[tbp]
\[\begin{array}{c}
\includegraphics[width=5cm]{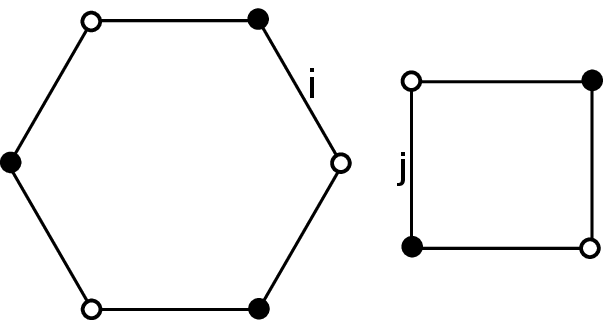}
\end{array}
\rightarrow
\begin{array}{c}
 \includegraphics[width=5cm]{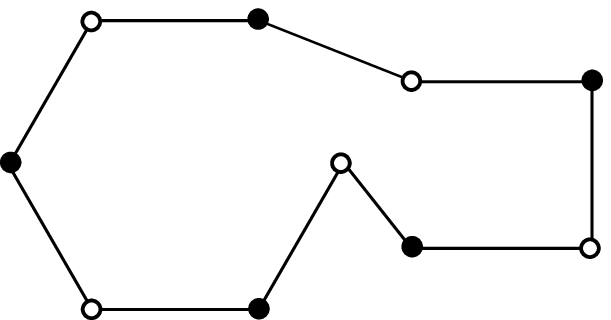}
\end{array}
\]
\caption{$\loops(S_0,S_1)$ and $\loops(S'_0,S_1)$ in the first case of proof
of Lemma \ref{LemChangingSignInvolution}.}
\label{FigJoinLoops}
\end{figure}

    \item Otherwise $i$ and $j$ are in the same loop $L$ of the graph
        $\loops(S_0,S_1)$.
        When we erase the edges $i$ and $j$ in this graph, the loop $L$ is
	split into two components $L_1$ and $L_2$.
        Let us say that $L_1$ contains the black extremity of $i$.
        By Lemma \ref{LemEvenDistance}, there is an even distance between $i$
        and $j$.
        This implies that the white extremity of $j$ is also in $L_1$, while
        its black extremity and the white extremity of $i$ are both in $L_2$.
        Therefore, when we add edges to obtain $\loops(S'_0,S_1)$, we do
        not change the number of connected components.

        Finally, the graph $\loops(S'_0,S_1)$ has one more connected component
        than $\loops(S_0,S_1)$ and the part \ref{lemma-part-b} of the lemma is
        also true in this case.
 
         This case is illustrated on Figure \ref{FigCutLoop}. \qedhere
\end{itemize}
\end{proof}

        \begin{figure}[tbp]
\[\begin{array}{c}
\includegraphics[width=3.5cm]{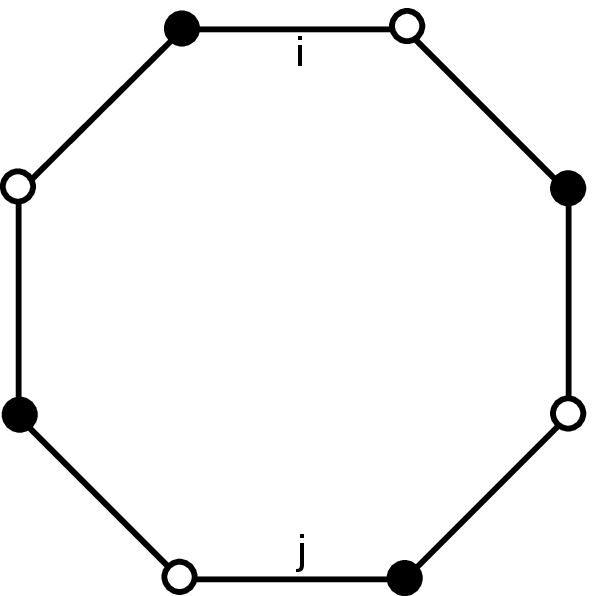}
\end{array}
\rightarrow
\begin{array}{c}
 \includegraphics[width=3.5cm]{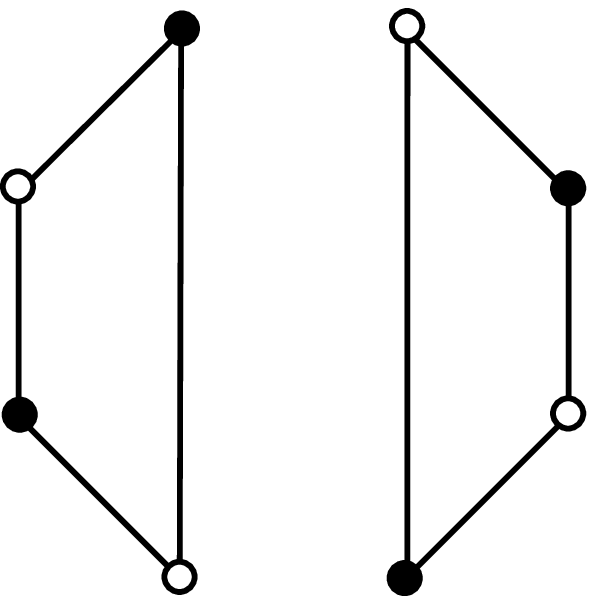}
\end{array}
\]
\caption{$\loops(S_0,S_1)$ and $\loops(S'_0,S_1)$ in the first case of proof
of Lemma \ref{LemChangingSignInvolution}.}
\label{FigCutLoop}
\end{figure}

From the discussion above it is clear that the lemma allows us to group
the terms in \eqref{eq:contrib_f} into canceling pairs.
Thus \eqref{eq:contrib_f} is equal to $0$ for any non-injective function $f$,
which implies that
\begin{multline*}
    \frac{1}{2^{\ell(\mu)}}
          \sum_{\substack{S_0 \text{ pair-partition} \\ \text{of
              } \{1,\dots,2|\mu|\} }} 
                  (-1)^{\loops(S_0,S_1)} \widehat{N}^{(2)}_{S_0,S_1,S_2} \\
 = \frac{1}{2^{\ell(\mu)}}
           \sum_{\substack{S_0 \text{ pair-partition} \\ \text{of
               } \{1,\dots,2|\mu|\} }} 
                   (-1)^{\loops(S_0,S_1)} {N}^{(2)}_{S_0,S_1,S_2}.
               \end{multline*}
Using Eq.~\eqref{eq:sigma_N_inj}, this proves
Theorem \ref{theo:ZonalFeraySniady-A2},
which is equivalent to Theorem \ref{theo:ZonalFeraySniady-A1}.

\subsection{Number of functions and Stanley's coordinates}
\label{subsec:N-in-terms}
In this paragraph we express the $N$ functions in terms of Stanley's coordinates
$\pp$ and $\qq$. This is quite easy and shows the equivalence between Theorems
\ref{theo:ZonalFeraySniady-A1} and \ref{theo:ZonalFeraySniady-B}.


\begin{lemma}\label{lem:N1_pq}
    Let $(S_0,S_1,S_2)$ be a triplet of pair-partitions.
    We will view the graphs $\loops(S_0,S_1)$ and $\loops(S_0,S_2)$ as the sets
    of their connected components. One has:
    \[N^{(1)}_{S_0,S_1,S_2}(\pp \times \qq) =
    \sum_{\varphi : \loops(S_0,S_2) \to \N^\star}
        \prod_{\ell \in \loops(S_0,S_2)} p_{\varphi(\ell)} 
        \prod_{m \in \loops(S_0,S_1)} q_{\psi(m)},\]
    where $\psi(m)=\max_\ell \varphi(\ell)$, with $\ell$ running over loops in
    $\loops(S_0,S_2)$, which have an edge with the same label as some edge of
    $m$.
\end{lemma}

\begin{proof}
    Fix a triplet $(S_0,S_1,S_2)$ of pair-partitions and sequences $\pp$
    and $\qq$.
    We set $\lambda=\pp \times \qq$ as in Section \ref{SubsectZonalCharPQ}.
    Let $g :  [2k] \to \lambda$ be a function verifying conditions 
    \ref{cond:Q0}, \ref{cond:Q1} and \ref{cond:Q2}.
    As $g$ fulfills \ref{cond:Q0} and \ref{cond:Q2},
    all elements $i$ in a given loop $\ell \in \loops(S_0,S_2)$
    have their image by $g$ in the same row $r_\ell$.
    We define $\varphi(\ell)$ as the integer $i$ such that
    \begin{equation}\label{IneqRB}
        p_1 + \cdots + p_{i-1} < r_\ell \leq p_1 + \cdots + p_i.
    \end{equation}
    This associates to $g$ a function $\varphi : \loops(S_0,S_2) \to \N^\star$.
    
    Let us fix a function $\varphi : \loops(S_0,S_2) \to \N^\star$.
    We want to find its pre-images $g: [2k] \to \lambda$. We have the following
    choices to make:
    \begin{itemize}
        \item we have to choose, for each loop $\ell \in \loops(S_0,S_2)$,
            the value of $r_\ell$.
            Due to inequality \eqref{IneqRB},
            one has $p_{\varphi(\ell)}$ choices for
            each loop $\ell$;
        \item then we have to choose, for each loop $m \in \loops(S_0,S_1)$,
            the value of $c_m$, the index of the common column of the images by
            $g$ of elements in $m$ (as we want $g$ to fulfill conditions
            \ref{cond:Q0} and \ref{cond:Q1}, all images of elements in $m$
            must be in the same column).
            By definition of $\psi(m)$, there is an integer $i \in m$, which
            belongs to a loop $\ell \in \loops(S_0,S_2)$ with 
            $\varphi(\ell)=\psi_m$.
            The image of $i$ by $g$ is the box $(r_\ell,c_m)$.
            As the $r_\ell$-th row of the diagram $\lambda$ has
            $q_{\varphi(\ell)}$ boxes, one has
            \begin{equation}\label{IneqCM}
                c_m \leq q_{\varphi(\ell)}.
            \end{equation}
            Finally, for each loop $m \in \loops(S_0,S_1)$,
            one has $q_{\psi(m)}$ possible values of $c_m$.
        \item A function $g :  [2k] \to \lambda$ verifying \ref{cond:Q0},
            \ref{cond:Q1} and \ref{cond:Q2} is uniquely determined
            by the two collections of numbers
            $(c_m)_{m \in \loops(S_0,S_1)}$ and 
            $(r_\ell)_{\ell \in \loops(S_0,S_2)}$.
            Indeed, if $i\in[2k]$, its image by $g$
            is the box $(r_\ell,c_m)$, where $m$ and $\ell$ are the loops of
            $\loops(S_0,S_1)$ and $\loops(S_0,S_2)$ containing $i$.
   \end{itemize}
   Conversely, if we choose two sequences of numbers
   $(c_m)_{m \in \loops(S_0,S_1)}$ and $(r_\ell)_{\ell \in \loops(S_0,S_2)}$
   fulfilling inequalities \eqref{IneqRB} and \eqref{IneqCM}, this defines a
   unique function $g$ fulfilling \ref{cond:Q0}, \ref{cond:Q1} and \ref{cond:Q2}
   associated to $\varphi$.
   It follows that each function $\varphi : \loops(S_0,S_2) \to \N^\star$ has
   exactly
   \[\prod_{\ell \in \loops(S_0,S_2)} p_{\varphi(\ell)} 
           \prod_{m \in \loops(S_0,S_1)} q_{\psi(m)},\]
   pre-images and the lemma holds.
\end{proof}

The above lemma shows that Theorem \ref{theo:ZonalFeraySniady-A1}
implies Theorem \ref{theo:ZonalFeraySniady-B}.

\begin{proof}[Proof of Theorem \ref{theo:ZonalFeraySniady-B}]
It is a direct application of Theorem \ref{theo:ZonalFeraySniady-A1} and of
the expression of $N^{(1)}_{S_0,S_1,S_2}$ in terms of Stanley's coordinates
that we establish in Lemma \ref{lem:N1_pq}.
\end{proof}

\subsection{Action of the axial symmetry group}
\label{subsec:integers_in_Lassalle}
\label{SubsectStanleyLassalle2}

The purpose of this paragraph is to prove
Proposition~\ref{PropStanleyLassalle2}. 

Theorem \ref{theo:ZonalFeraySniady-B} implies that the coefficients of
$(-1)^k \Sigma^{(2)}_\mu(\pp,-\qq)$ are non-negative.
But it is not obvious from this formula that the coefficients are integers.
We will prove it in this paragraph by grouping some identical terms in 
Theorem~\ref{theo:ZonalFeraySniady-A1} before applying Lemma \ref{lem:N1_pq}.

The following lemma will be useful to find some identical terms.
\begin{lemma}
    Let $(S_0,S_1,S_2)$ be a triplet of pair-partitions of $[2k]$ and
    $\sigma$ be a permutation in $\Sym{2k}$. Then
    \[N^{(1)}_{(\sigma \cdot S_0, \sigma \cdot S_1, \sigma \cdot S_2)}=
    N^{(1)}_{(S_0,S_1,S_2)}.\]
    \label{LemInvN}
\end{lemma}
\begin{proof}
    Map $f : [2k] \to 2\lambda$ satisfies conditions \ref{cond:Q0},
    \ref{cond:Q1} and \ref{cond:Q2} with respect to
    $(\sigma \cdot S_0, \sigma \cdot S_1, \sigma \cdot S_2)$
    if and only if
    $f \circ \sigma$ satisfies conditions \ref{cond:Q0}, \ref{cond:Q1} and
    \ref{cond:Q2} with respect to $(S_0,S_1,S_2)$.
\end{proof}

From now on, we fix a partition $\mu$ of $k$ and a couple $(S_1,S_2)$ of
pair-partitions of $[2k]$ of type $\mu$.

Choose arbitrarily an edge $j_{i,1}$ in each loop $L_i$ (which is of length $2\mu_i$).
Denote $j_{i,2}=S_2(j_{i,1})$, $j_{i,3}=S_1(j_{i,2})$ and so on until
$j_{i,2\mu_i}=S_2(j_{i,2\mu_i-1})$, which fulfills $S_1(j_{i,2\mu_i})= j_{i,1}$.
We consider the permutation $r_i$ in $\Sym{2k}$ which sends $j_{i,m}$ to
$j_{i,2\mu_i+1-m}$ for any $m \in [2\mu_i]$ and fixes all other integers.
Geometrically, the sequence $(j_{i,m})_{m \in [2\mu_i]}$  is obtained by
reading the labels of the edges along the loop $L_i$
and $r_i$ is an axial symmetry of the loop $L_i$.
\begin{itemize}
    \item $r_i$ permutes the black vertices of the graph $\loops(S_1,S_2)$ 
        (it is an axial symmetry of $L_i$ and fixes the elements
        of the other connected components). It means that $r_i\cdot S_1=S_1$.
        
        In the same way, it permutes the white vertices therefore
        $r_i\cdot S_2 =S_2$.

    \item Permutations $r_i$ are of order $2$ and they clearly commute with each
	other (their supports are pairwise disjoint);
        therefore, they generate a subgroup $G$ of order $2^{\ell(\mu)}$ of
        $\Sym{2|\mu|}$. Moreover, for a fixed integer $j$, the orbit
        $\{g(j): g \in G\}$ contains exactly two elements: $j$ and $r_i(j)$,
        where $i$ is the index of the loop of $\loops(S_1,S_2)$ containing $i$.
\end{itemize}

Using Lemma \ref{LemInvN}, for any pair-partition $S_0$, one has
\[N^{(1)}_{g \cdot S_0,S_1,S_2} = N^{(1)}_{g \cdot S_0,g \cdot S_1,g \cdot S_2}
= N^{(1)}_{S_0,S_1,S_2},\]
where $g$ is equal to any one of the $r_i$.
It immediately extends to any $g$ in $G$.
In the same way, we have 
\[(-1)^{\loops(g \cdot S_0,S_1)} =(-1)^{\loops(g \cdot S_0,g \cdot S_1)}
=(-1)^{\loops(S_0,S_1)}. \]
Therefore Theorem~\ref{theo:ZonalFeraySniady-A1} can be restated as:
\begin{equation}\label{ZonalCharOrbits}
\Sigma^{(2)}_\mu = \sum_{\Omega\text{ orbits} \atop \text{under }G}
(-1)^{\loops(S_0(\Omega),S_1)}
\frac{2^{|\loops(S_0(\Omega),S_1)|}}{2^{\ell(\mu)}}
|\Omega|\ N^{(1)}_{S_0(\Omega),S_1,S_2},
\end{equation}
where the sum runs over the orbits $\Omega$ of the set of all pair-partitions
of $[2k]$ under the action of $G$ and
where $S_0(\Omega)$ is any element of the orbit $\Omega$.

\begin{lemma}
    For each orbit $\Omega$ of the set of pair-partitions of $[2k]$ under the
    action of $G$, the quantity
    \[ \frac{2^{|\loops(S_0(\Omega),S_1)|}}{2^{\ell(\mu)}} |\Omega|\]
    is an integer.
    \label{LemInteger}
\end{lemma}
This lemma and Eq.~\eqref{ZonalCharOrbits} imply 
Proposition~\ref{PropStanleyLassalle2} (because the $N$ functions are 
polynomials with integer coefficients in variables $\pp$ and $\qq$,
see Lemma \ref{lem:N1_pq}).

\begin{proof}
Let us fix an element $S_0=S_0(\Omega)$ in the orbit $\Omega$.
The quotient $\frac{2^{\ell(\mu)}}{|\Omega|}$ is the cardinality of the
stabilizer $\Stab(S_0) \subset G$ of $S_0$. Therefore it divides the cardinality
of $G$, which is $2^{\ell(\mu)}$, and, hence is a power of $2$.
Besides, any permutation $\pi\in\Stab(S_0) \subset G$ leaves $S_0$ and $S_1$
invariant hence $\pi$ is entirely determined by the its values on 
$\{e_L: L \in \loops(S_0,S_1)\}$, where $e_L$ is an arbitrary element in the loop $L$
(the argument is the same as in the proof of Lemma \ref{LemCountCouples}).
As each integer, and in particular each $e_L$, has only two possible images
by the elements of $G$, this implies that the cardinality of $\Stab(S_0)$ is
smaller or equal to $2^{|\loops(S_0,S_1)|}$.
But it is power of $2$ so $|\Stab(S_0)|=\frac{2^{\ell(\mu)}}{|\Omega|}$
divides $2^{|\loops(S_0,S_1)|}$.
\end{proof}

We will give now an alternative way to end the proof,
which is less natural but more meaningful from the combinatorial point of view.
As before, the partition $\mu \vdash k$ is fixed,
as well as a couple $(S_1,S_2)$ of pair-partitions of $[2k]$ of type $\mu$.
We call an orientation $\orient$ of the elements in $[2k]$, the choice,
for each number in $[2k]$, of a color (red or green).

If $S_0$ is a pair-partition, we say that an orientation $\orient$
is compatible with the loops $\loops(S_0,S_1)$ if
each pair of $S_0$ and each pair of $S_1$ contains one red and one green element.
We denote by $\pairings^o$ the set of couples $(S_0,\orient)$
such that $\orient$ is compatible with $\loops(S_0,S_1)$.

In such an orientation, the color of an element $e_L$ in a loop $L \in \loops(S_0,S_1)$
determines the colors of all elements in this loop.
Nevertheless, the colors of the $\{e_L, L \in \loops(S_1,S_2)\}$, where
$e_L$ is an arbitrary element of $L$, can be chosen idependently.
Therefore, for a given pair-partition $S_0$,
there are exactly $2^{|\loops(S_0,S_1)|}$ orientations compatible with $\loops(S_0,S_1)$.
Hence, Theorem~\ref{theo:ZonalFeraySniady-A1} can be rewritten as:
\begin{equation}\label{EqOrientations}
        \Sigma^{(2)}_\mu =\frac{1}{2^{\ell(\mu)}}
                \sum_{(S_0,\orient)} (-1)^{\loops(S_0,S_1)}\
                   N^{(1)}_{S_0,S_1,S_2},
\end{equation}
where the sum runs over $\pairings^o$.

Of course, the group $\Sym{2k}$, and hence its subgroup $G$,
acts on the set of orientations of $[2k]$.
By definition, if $\orient$ is an orientation and $\sigma$ a permutation,
the color given to $\sigma(i)$ in the orientation $\sigma \cdot \orient$
is the color given to $i$ in $\orient$.

We will consider the diagonal action of $G$ on couples $(S_0,\orient)$.
It is immediate that this action preserves $\pairings^o$.

\begin{lemma}
    The diagonal action of $G$ on $\pairings^o$ 
    is faithful.
\end{lemma}

\begin{proof}
    Let us suppose that $g \cdot (S_0,\orient)=(S_0,\orient)$.
    We use the definition of the integers $j_{i,m}$ given at
    the beginning of the paragraph to define the group $G$.
    Recall that $S_1$ contains, for each $i$,
    the pair $\{j_{i,1},j_{i,2\mu_i}\}$.
    Hence, as $\orient$ is compatible with $\loops(S_0,S_1)$,
    the integers $j_{i,1}$ and $j_{i,2\mu_i}$ have different colors in $\orient$.
    But $\orient$ is fixed by $g$, so $g(j_{i,1})$ cannot be equal
    to $j_{i,2\mu_i}$.
    This means that $g$ does not act like the mirror symmetry $r_i$
    on the loop $L_i$; hence $g$ acts on the loop $L_i$ like the identity.
    As this is true for all loops in $\loops(S_1,S_2)$,
    the permutation $g$ is equal to the identity.
\end{proof}

Finally, as $N^{(1)}_{g \cdot S_0,S_1,S_2}= N^{(1)}_{S_0,S_1,S_2}$, we
can group together in Eq.~\eqref{EqOrientations} the terms corresponding
to the $2^{\ell(\mu)}$ couples $(S_0,\orient)$ in the same orbit.
We obtain the following result.
\begin{theorem}\label{ThmOrientationsOrbits}
 Let $\mu$ be a partition of the integer $k$ and $(S_1, S_2)$ be a fixed
    couple of pair-partitions of\/ $[2k]$ of type $\mu$. Then,
\begin{equation}\label{EqOrientationsOrbits}
        \Sigma^{(2)}_\mu =
                \sum_{\Omega} (-1)^{\loops(S_0(\Omega),S_1)}\
                   N^{(1)}_{S_0(\Omega),S_1,S_2},
\end{equation}
where the sum runs over orbits $\Omega$ of $\pairings^o$ under the action of $G$
(for such an orbit, $S_0(\Omega)$ is the first element of an arbitrary couple
in $\Omega$).
\end{theorem}

Using Lemma~\ref{lem:N1_pq}, this formula gives an alternative proof of
Proposition~\ref{PropStanleyLassalle2}.
From a combinatorial point of view, it is more satisfying than the one above
because we are unable to interpret the number 
$\frac{2^{|\loops(S_0(\Omega),S_1)|}}{2^{\ell(\mu)}} |\Omega|$ in 
Eq~\eqref{ZonalCharOrbits}.
More details are given in Section \ref{SubsecGluingsNonoriented}.

\begin{remark}\label{RemSchurOrient}
    Let us consider orientations $\orient$ compatible with $\loops(S_0,S_1)$
    and $\loops(S_0,S_2)$.
    Each such an orientation can be viewed as a partition of $[2k]$ into two sets of size $k$, such that
    each pair in $S_0$, $S_1$ or $S_2$ contains an element of each set.
    If such a partition is given, the pair-partitions $S_0$, $S_1$ and $S_2$
    can be interpreted as permutations and the Schur case can be formulated
    in these terms (see Remark~\ref{rem:analogue_schur}).
\end{remark}

\section{Kerov polynomials}\label{sec:kerov}

\subsection{Graph associated to a triplet of pair-partitions}
Let $(S_0,S_1,S_2)$ be a triplet of pair partitions of $[2k]$.
We define the bipartite graph $G(S_0,S_1,S_2)$ in the following way.
\begin{itemize}
\item Its set of black vertices is $\loops(S_0,S_1)$.
\item Its set of white vertices is $\loops(S_0,S_2)$.
\item There is an edge between a black vertex $\ell \in \loops(S_0,S_1)$
    and a white vertex $\ell' \in \loops(S_0,S_2)$
if (and only if) the corresponding subsets of $[2k]$ have a non-empty intersection.
\end{itemize}

Note that the connectivity of $G(S_0,S_1,S_2)$ corresponds exactly to condition
 \ref{enum:first-condition} of Theorem \ref{theo:kerov}.

This definition is relevant because 
the function $N^{(1)}_{S_0,S_1,S_2}$ depends only on the graph $G(S_0,S_1,S_2)$.
Indeed, let us define, for any bipartite graph $G$,
a function $N^{(1)}_G$ on Young diagram as follows:
\begin{definition}\label{DefNG}
Let $G$ be a bipartite graph and $\lambda$ a Young diagram.
We denote $N^{(1)}_G(\lambda)$ the number of functions $f$
\begin{itemize}
  \item sending black vertices of $G$ to the set of column indices of $\lambda$;
  \item sending white vertices of $G$ to the set of row indices of $\lambda$;
  \item such that, for each edge of $G$ between a black vertex $b$ and a white vertex $w$,
  the box $(f(w),f(b))$ belongs to the Young diagram $\lambda$
  ({\it i.e.} $1\leq f(b) \leq \lambda_{f(w)}$).

\end{itemize}
\end{definition}
Then, using the arguments of the proof of Lemma~\ref{lem:N1_pq}, one has:
\[ N^{(1)}_{S_0,S_1,S_2} = N^{(1)}_{G(S_0,S_1,S_2)}. \]
As characters and cumulants, $N^{(1)}_G$ can be defined on non-integer
stretching of Young diagrams using Lemma~\ref{lem:N1_pq}.

\subsection{General formula for Kerov polynomials}
Our analysis of zonal Kerov polynomials will be based on the following general
result.
\begin{lemma}
\label{lem:extract-kerov}
Let $\mathcal{G}$ be a finite collection of connected bipartite graphs and let
$\mathcal{G}\ni G \mapsto m_G$ be a scalar-valued function on it. We assume that
$$ F(\lambda) = \sum_{G\in\mathcal{G}} m_G N^{(1)}_G(\lambda) $$
is a polynomial function on the set of Young diagrams; in other words $F$ can
be expressed as a polynomial in free cumulants.

Let $s_2,s_3,\dots$ be a sequence of non-negative
integers with only finitely many non-zero elements; then
$$ \left[ R_2^{s_2} R_3^{s_3} \cdots\right] F =
(-1)^{s_2+2s_3+3s_4+\cdots+1} \sum_{G\in\mathcal{G}} \sum_{q}  m_G,$$
where the sums runs over $G\in\mathcal{G}$ and $q$ such that:
\begin{enumerate}[label=(\alph*)]
 \addtocounter{enumi}{1}
 \item
\label{enum:ilosc2_Graphs}
the number of the black vertices of $G$ is
equal to $s_2+s_3+\cdots$;
 \item
\label{enum:boys-and-girls_Graphs}
the total number of vertices of $G$ is equal to $2 s_2+3 s_3+4 s_4+\cdots$;
 \item
\label{enum:kolorowanie_Graphs}
$q$ is a function from the set of
the black vertices to the set $\{2,3,\dots\}$; we require that each number
$i\in\{2,3,\dots\}$ is used exactly $s_i$ times;
\item
\label{enum:marriage_Graphs}
for every subset $A\subset V_\circ(G)$ of black vertices of $G$
which is nontrivial (i.e., $A\neq\emptyset$ and $A\neq V_\circ(G)$) there are
more than $\sum_{v\in A} \big( q(v)-1 \big)$ white vertices which are
connected to at least one vertex from $A$.
\end{enumerate}
\end{lemma}

This result was proved in our previous paper with
Dołęga \cite{DolegaF'eray'Sniady2008}
in the special case when
$F=\Sigma^{(1)}_n$ and $\mathcal{G}$ is the (signed) collection of bipartite maps
corresponding to all factorizations of a cycle, however it is not difficult to
verify that the proof presented there works without any modifications also in
this more general setup.


\subsection{Proof of Theorem \ref{theo:kerov}}
\begin{proof}[Proof of Theorem \ref{theo:kerov}]
We consider for simplicity the case when $\mu=(k)$ has only one part.
By definition, it is obvious that, for any $G$ and $\lambda$,
\[N^{(1)}_G(\alpha \lambda) = \alpha^{|V_\bullet(G)|} N^{(1)}_G(\lambda),\]
where $|V_\bullet(G)|$ is the number of black vertices of $G$. Hence,
Theorem \ref{theo:ZonalFeraySniady-A1} can be rewritten in the form
$$  F(\lambda):=  \Sigma^{(2)}_k \left(\frac{1}{2}\lambda\right)=
\frac{1}{2}
                \sum_{S_0} (-1)^{k+|\loops(S_0,S_1)|}\
                   N^{(1)}_{S_0,S_1,S_2} (\lambda).
$$
Function $F$ is a polynomial function on the set of Young diagrams
\cite[Proposition 2]{Lassalle2008a}.
As the involutions corresponding to $S_1$ and $S_2$ span a transitive
subgroup of $\Sym{2k}$ (because the couple $(S_1,S_2)$ has type $(k)$),
the graph corresponding to $S_0,S_1,S_2$ is connected and Lemma
\ref{lem:extract-kerov} can be applied. 
$$ \left[ R_2^{s_2} R_3^{s_3} \cdots\right] F =\frac{1}{2} (-1)^{1+k+
 |\loops(S_0,S_1)|+s_2+2 s_3+3 s_4+\cdots}  \sum_{S_0} \sum_{q}  1,$$
where the sum runs over $S_0$ and $q$ such that the graph
$G(S_0,S_1,S_2)$ and $q$ fulfill the assumptions of Lemma
\ref{lem:extract-kerov}.
Notice that, for such a $S_0$, the number $|\loops(S_0,S_1)|$
of black vertices of $G(S_0,S_1,S_2)$ is $s_2+ s_3+ s_4$.
Under a change of variables $\widetilde{\lambda}=\frac{1}{2}\lambda$ we have
$ \Sigma^{(2)}_{k}(\widetilde{\lambda}) = F(\lambda) $
and $R_i=R_i(\lambda)=2^i R_i^{(2)}(\widetilde{\lambda})$ and thus
\begin{align*}
    \left[ \big(R^{(2)}_2\big)^{s_2} \big(R^{(2)}_3\big)^{s_3} \cdots\right] \Sigma^{(2)}_{k}
&= 2^{2 s_2 + 3 s_3 + \ldots} \left[ R_2^{s_2} R_3^{s_3} \cdots\right] F \\
&=(-1)^{1+k+2s_2+3s_3+\cdots} 2^{-1+2 s_2 + 3 s_3 + \ldots} \mathcal{N},
\end{align*}
where $\mathcal{N}$ is the number of couples $(S_0,q)$ as above.
This ends the proof in the case $\mu=(k)$.

Consider now the general case $\mu=(k_1,\dots,k_\ell)$. In an analogous way as
in \cite[Theorem 4.7]{DolegaF'eray'Sniady2008} one can show that $\kappa^{\text{id}}(
\Sigma^{(\alpha)}_{k_1}, \dots, \Sigma^{(\alpha)}_{k_\ell} )$ is equal to the
right-hand side of \eqref{eq:ZonalFeraySniady-A1}, where $S_1,S_2$ are chosen so
that $\type(S_1,S_2)=\mu$ and the summation runs over $S_0$ with the property
that the corresponding graph $G(S_0,S_1,S_2)$ is connected.
Therefore
\begin{multline*} F(\lambda):= (-1)^{\ell-1} \kappa^{\text{id}}(
\Sigma^{(\alpha)}_{k_1}, \dots, \Sigma^{(\alpha)}_{k_\ell} )
\left(\frac{1}{2}\lambda\right)= \\
\frac{1}{2^{\ell(\mu)}} (-1)^{\ell-1}
                \sum_{S_0} (-1)^{|\mu|+|\loops(S_0,S_1)|}\
                   N^{(1)}_{S_0,S_1,S_2} (\lambda).
\end{multline*}
The remaining part of the proof follows in an analogous way.
\end{proof}

\subsection{Particular case of Lassalle conjecture for Kerov polynomials}
\label{SubsectKerovLassalle}

The purpose of this paragraph is to prove Proposition~\ref{PropKerovLassalle2},
which states that the coefficients 
\[ \left[ \big(R^{(2)}_2\big)^{s_2} \big(R^{(2)}_3\big)^{s_3} \cdots\right] K^{(2)}_{\mu} \]
are integers.

This does not follow directly from Theorem~\ref{theo:kerov} because of the factor
$2^{\ell(\mu)}$. As in Section \ref{subsec:integers_in_Lassalle} we will use
Theorem~\ref{ThmOrientationsOrbits}.
With the same argument as in the previous paragraph, one obtains the following
result:
\begin{theorem}\label{LemKerovOrbits}
 Let $\mu$ be a partition of the integer $k$ and $(S_1, S_2)$ be a fixed
    couple of pair-partitions of\/ $[2k]$ of type $\mu$.
Let $s_2,s_3,\dots$ be a sequence of non-negative
integers with only finitely many non-zero elements.

Then the rescaled coefficient
\[ (-1)^{|\mu|+\ell(\mu)+2s_2+3s_3+\cdots}\                                                                                  
2^{- (s_2+2s_3+3s_4+\cdots)}             
\left[ \big(R^{(2)}_2\big)^{s_2} \big(R^{(2)}_3\big)^{s_3} \cdots\right] K^{(2)}_{\mu} \]
 is equal to the number of orbits $\Omega$ of couples $(S_0,\phi)$
 in $\pairings^o$ under the action of $G$,
 such that any element $S_0(\Omega)$ of this orbit fulfills conditions
 \ref{enum:first-condition}, \ref{enum:ilosc2}, \ref{enum:boys-and-girls},
\ref{enum:kolorowanie} and \ref{enum:marriage} of Theorem \ref{theo:kerov}.
\end{theorem}

This implies immediately Proposition~\ref{PropKerovLassalle2}.
In fact, one shows a stronger result, which fits with Lassalle's data:
the coefficient of $\big(R^{(2)}_2\big)^{s_2} \big(R^{(2)}_3\big)^{s_3} \cdots$
in $K^{(2)}_{\mu}$ is a multiple of $2^{s_2+2s_3+3s_4+\cdots}$.

\section{Maps on possibly non orientable surfaces}\label{SectCombinatorics}

The purpose of this section is to emphasize the fact that triplets of
pair-partitions are in fact a much more natural combinatorial object than it may
seem at the first glance: each such a triple can be seen as a graph drawn on a
(non-oriented) surface.

\subsection{Gluings of bipartite polygons}\label{SubsectGluings}

It has been explained in Section \ref{SubsubsectCouplePP} how a couple of
pair-partitions $(S_1,S_2)$ of the same set $[2k]$ can be represented by the
collection $\loops(S_1,S_2)$ of edge-labeled polygons: the white (resp.~black)
vertices correspond to the pairs of $S_1$ (resp.~$S_2$). For instance, let us
consider the couple
 \begin{eqnarray*}
        S_1 &= \big\{
\{1,15\},\{2,3\},\{4,14\},\{13,16\},\{5,7\},\{6,10\},\{8,11\},\{9,12\}\big\};\\
        S_2 &= \big\{
\{1,10\},\{2,7\},\{8,13\},\{9,14\},\{3,5\},\{4,12\},\{6,15\},\{11,16\}\big\}.\\
    \end{eqnarray*}
The corresponding polygons are drawn on Figure \ref{ExPolygons2}.

    \begin{figure}[tbp]
        \[\begin{array}{c} \includegraphics[height=1.8cm]{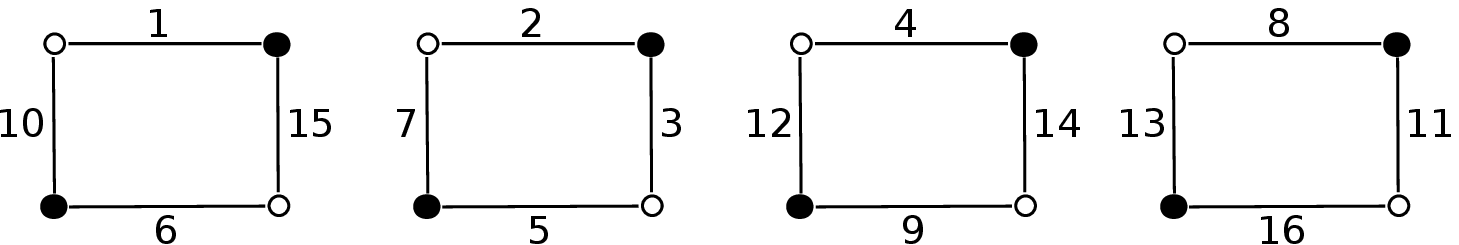} \end{array}\]
        \caption{Polygons associated to the couple $(S_1,S_2)$.}
        \label{ExPolygons2}
    \end{figure}

With this in mind, one can see the third pair-partition $S_0$ as a set of
instructions to glue the edges of our collection of polygons. If $i$ and $j$ are
partners in $S_0$, we glue the edges labeled $i$ and $j$ together in such a way
that their black (respectively, white) extremities are glued together. When
doing this, the union the polygons becomes a (non-oriented, possibly
non-connected) surface, which is well-defined up to continuous deformation of
the surface. The border of the polygons becomes a bipartite graph drawn on
this surface (when it is connected, this object is usually called \emph{map}).
We denote $M(S_0,S_1,S_2)$ the union of maps obtained in this way.
An edge of $M(S_0,S_1,S_2)$ is formed by two edge-sides, each one of them
corresponding to an edge of a polygon.

For instance, we continue the previous example by choosing
\[        S_0 = \big\{
\{1,2\},\{3,4\},\{5,6\},\{7,8\},\{9,10\},\{11,12\},\{13,14\},\{15,16\}\big\}.\]
We obtain a graph drawn on a Klein bottle, represented on the left-hand side of
Figure \ref{FigMap} (the Klein bottle can be viewed as the square with some
identification of its edges).
A planar representation of this map, involving artificial crossings and twists
of edges,
is given on the right-hand side of the same figure.
    \begin{figure}[tbp]
        \[\begin{array}{c} \includegraphics[width=5.5cm]{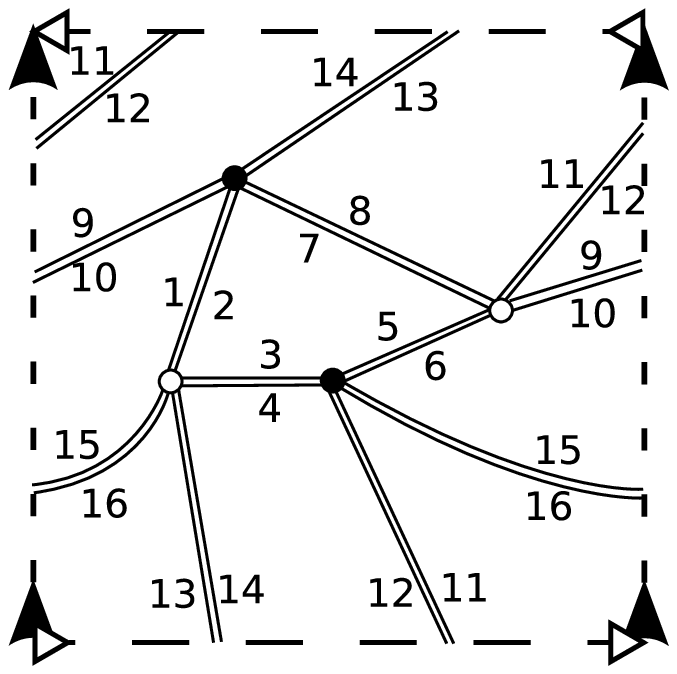} \end{array}
            \quad
        \begin{array}{c} \includegraphics[width=5.5cm]{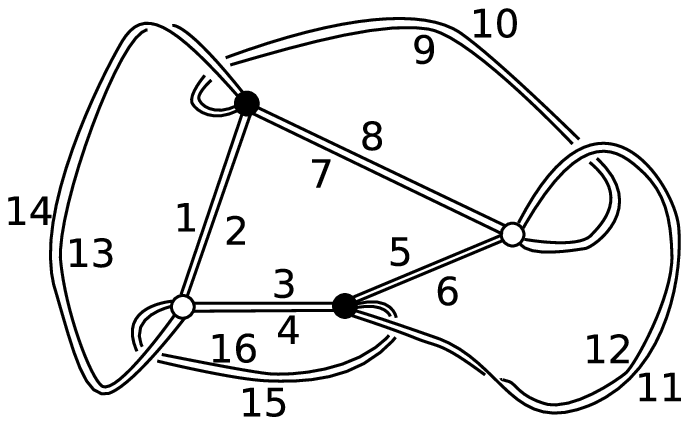}\end{array}\]
        \caption{Example of a labeled map on Klein bottle.}
        \label{FigMap}
    \end{figure}

    \subsection{The underlying graph of a gluing of polygons}\label{SubsectUnderlyingGraph}
By definition, the black vertices of $\loops(S_1,S_2)$ correspond to the pairs in $S_1$.
If $\{i,j\}$ is a pair in $S_0$, when we glue the edges $i$ and $j$ together,
we also glue the black vertex containing $i$ with the black vertex containing $j$.
Hence, when all pairs of edges have been glued, we have one black vertex
per loop in $\loops(S_0,S_1)$.

In the same way, the white vertices of the union of maps $M(S_0,S_1,S_2)$ correspond
to the loops in $\loops(S_0,S_2)$.

The edges of the union of maps correspond to pairs in $S_0$, therefore a black vertex 
$\ell \in \loops(S_0,S_1)$ is linked to a white vertex $\ell' \in \loops(S_0,S_2)$
if there is a pair of $S_0$ which is included in both $\ell$ and $\ell'$.
As $\ell$ and $\ell'$ are unions of pairs of $S_0$, this is equivalent to the 
fact that they have a non-empty intersection.

Hence the underlying graph of $M(S_0,S_1,S_2)$ ({\it i.e.} the graph obtained
by forgetting the surface, the edge labels and the multiple edges) is exactly
the graph $G(S_0,S_1,S_2)$ defined in Section \ref{DefNG}.

It is also interesting to notice (even if it will not be useful in this paper)
that the faces of
the union of maps $M(S_0,S_1,S_2)$ (which are, by definition, the connected components of 
the surface after removing the graph) correspond by construction to the loops in
$\loops(S_1,S_2)$.

\begin{remark}
The related combinatorics of maps which are not bipartite has been studied by
Goulden and Jackson \cite{Goulden1996a}.
\end{remark}

\subsection{Reformulation of Theorems \ref{theo:ZonalFeraySniady-A1} and \ref{theo:kerov}}

In some of our theorems, we fix a partition $\mu \vdash k$
and a couple of pair-partitions $(S_1,S_2)$ of type $\mu$.
Using the graphical representation of Section \ref{SubsubsectCouplePP},
it is the same as fixing $\mu$ and a collection of edge-labelled polygons
of lengths $2\mu_1, 2\mu_2, \dots$.

In this context, the set of pair-partitions is the set of maps
obtained by gluing by pair the edges of these polygons (see
Section \ref{SubsectGluings}). 

Then the different quantities involved in our theorems have a
combinatorial translation: $G(S_0,S_1,S_2)$ is the underlying graph
of the map (Section~\ref{SubsectUnderlyingGraph}),
$\loops(S_0,S_1)$ the set of its black vertices and
$\loops(S_0,S_2)$ the set of its white vertices.

One can now give combinatorial formulations for two of our theorems.

\begin{theorem}\label{TheoStanleyCombi}
    Let $\mu$ be a partition of the integer $k$.
    Consider a collection of edge-labelled polygons of lengths
    $2\mu_1, 2\mu_2, \dots$.
    Then one has the following equality between functions on the set of Young
    diagrams:
    \begin{equation}
        \Sigma^{(2)}_\mu =\frac{(-1)^k}{2^{\ell(\mu)}}
                \sum_{M} (-2)^{|V_\bullet(M)|}\ N^{(1)}_{G(M)},
           \end{equation}
    where the sum runs over unions of maps obtained by gluing by pair the 
    edges of our collection of polygons in all possible ways;
    $|V_\bullet(M)|$ is the number of black vertices of $M$
    and $G(M)$ the underlying graph.
\end{theorem}

\begin{proof}Reformulation of Theorem~\ref{theo:ZonalFeraySniady-A1}.
\end{proof}

\begin{theorem}\label{TheoKerovCombiB}
    Let $\mu$ be a partition of the integer $k$.
Consider a collection of edge-labelled polygons of lengths $2\mu_1, 2\mu_2, \dots$. 

Let $s_2,s_3,\dots$ be a sequence of non-negative
integers with only finitely many non-zero elements.

The rescaled coefficient
$$  (-1)^{|\mu|+\ell(\mu)+2s_2+3s_3+\cdots} (2)^{\ell(\mu) - (2s_2+3s_3+\cdots)}
\left[ \left(R_2^{(2)}\right)^{s_2}
\left(R_3^{(2)}\right)^{s_3} \cdots\right]
K^{(2)}_\mu $$ of the (generalized)
zonal Kerov polynomial is equal to the number of pairs $(M,q)$
such that 
\begin{itemize}
  \item $M$ is a connected map obtained by gluing edges of
our polygons by pair;
   \item the pair $(G(M),q)$, where $G(M)$ is the underlying graph of $M$,
   fulfill conditions \ref{enum:ilosc2_Graphs}, \ref{enum:boys-and-girls_Graphs},
   \ref{enum:kolorowanie_Graphs} and \ref{enum:marriage_Graphs} of 
   Lemma \ref{lem:extract-kerov}.
\end{itemize}
\end{theorem}

\begin{proof}Reformulation of Theorem~\ref{theo:kerov}.
\end{proof}

\begin{remark}
As $G(M)$ is an unlabelled graph, the edge-labelling of the polygons is not important.
But we still have to consider a family of polygons without automorphism.
So, instead of edge-labelled polygons, we could consider a family of distinguishable
edge-rooted polygons (which means that each polygon has a marked edge and
that we can distinguish the polygons, even the ones with the same size).
\end{remark}

\begin{remark}
These results are analogues to results for characters of the symmetric groups.
The latter are the same (up to normalizing factors),
except that one has to consider a family of oriented polygons
and consider only gluings which respect this orientation
(hence the resulting surface has also a natural orientation).
These results can be found in papers \cite{F'eray'Sniady-preprint2007} and
\cite{DolegaF'eray'Sniady2008}, but, unfortunately, not under this formulation.
\end{remark}

\subsection{Orientations around black vertices}\label{SubsecGluingsNonoriented}

The purpose of this section is to give a combinatorial interpretation
of Theorem~\ref{ThmOrientationsOrbits} and Theorem~\ref{LemKerovOrbits}.

As before, $S_0$ is interpreted as a map obtained by gluing by pair
the edges of a collection of distinguishable edge-rooted polygons.

An orientation $\phi$ consists in orienting each edge of this
collection of polygons ({\it i.e.} each edge-side of the map).
It is compatible with $\loops(S_0,S_1)$ if, around each black vertex,
outgoing and incoming edge-sides alternate (see Figure \ref{FigBlackCompatible}).

    \begin{figure}[tbp]
        \includegraphics[height=2cm]{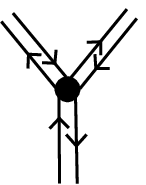} 
        \caption{A black vertex after a black-compatible orientation and gluing.}
        \label{FigBlackCompatible}
    \end{figure}

To make short, we will say in this case, that the orientation and
the gluing are black-compatible.
So $\pairings^o$ is the set of black-compatible orientations and gluings
of our family of polygons.

In our formulas we consider orbits of $\pairings^o$ under the action of $G$.
Recall that $G$ is the group generated by the $r_L$, for $L \in \loops(S_1,S_2)$
where $r_L$ is an axial symmetry of the loop $L$
(and its axis of symmetry goes through a black vertex).

Notice that, in general, combinatorial objects with unlabeled components are,
strictly speaking, equivalence classes of the combinatorial objects of the same
type with labeled components; the equivalence classes are the orbits of the
action of some group which describes the symmetry of the unlabeled version. 

In our case, a (bipartite) polygon with a marked edge has no symmetry.
But, if we consider a polygon with a marked black vertex, its automorphism
group is exactly the two-element group generated by the axial symmetry
going though this vertex.

Therefore, the orbits of $\pairings^o$ under the action of $G$ can be
interpreted as the black-compatible orientations and gluing of a collection
of distinguishable vertex-rooted polygons.

We can now reformulate Theorems~\ref{ThmOrientationsOrbits} and \ref{LemKerovOrbits}.

\begin{theorem}\label{TheoStanleyCombi2}
    Let $\mu$ be a partition of the integer $k$.
    Consider a collection of unlabeled polygons of lengths
    $2\mu_1, 2\mu_2, \dots$ with one marked black vertex per polygon.
    Then one has the following equality between functions on the set of Young
    diagrams:
    \[ 
        \Sigma^{(2)}_\mu
          =(-1)^k
          \sum_{\vec{M}} (-1)^{|V_\bullet(M)|}\ N^{(1)}_{G(M)},   
    \]
    where the sum runs over all unions of maps with oriented edge-sides
    obtained by a black-compatible orientation and gluing of the 
    edges of our collection of polygons;
    $M$ is the map obtained by forgetting the orientations of the edge-sides,
    $|V_\bullet(M)|$ is the number of black vertices of $M$
    and $G(M)$ the underlying bipartite graph.
\end{theorem}

\begin{theorem}\label{TheoKerovCombiA}
Let $\mu$ be a partition of the integer $k$.
Consider a collection of unlabeled polygons of lengths
$2\mu_1, 2\mu_2, \dots$ with one marked black vertex per polygon.
Let $s_2,s_3,\dots$ be a sequence of non-negative
integers with only finitely many non-zero elements.

Then the rescaled coefficient
$$  (-1)^{|\mu|+\ell(\mu)+2s_2+3s_3+\cdots}\ 
2^{- (s_2+2s_3+3s_4+\cdots)}
\left[ \left(R_2^{(2)}\right)^{s_2}
\left(R_3^{(2)}\right)^{s_3} \cdots\right]
K^{(2)}_\mu $$ of the (generalized)
zonal Kerov polynomial is equal to the number of pairs $(\vec{M},q)$
such that 
\begin{itemize}
    \item $\vec{M}$ is a connected map with oriented edge-sides
        obtained by a black-compatible orientation and gluing of the 
            edges of our collection of polygons;
            denote $M$ the map obtained by forgetting the orientations of the edge-sides.
   \item the pair $(G(M),q)$, where $G(M)$ is the underlying graph of $M$,
   fulfills conditions \ref{enum:ilosc2_Graphs},
    \ref{enum:boys-and-girls_Graphs},
   \ref{enum:kolorowanie_Graphs} and \ref{enum:marriage_Graphs} of 
   Lemma \ref{lem:extract-kerov}.
\end{itemize}
\end{theorem}

\begin{remark}
    It is easy to see that
    a black- and white-compatible orientation and gluing of a collection
    of polygons leads to a map on a oriented surface.
    Therefore the analogue results in the Schur case can be interpreted
    in these terms.

    This remark is the combinatorial version of Remark \ref{RemSchurOrient}.
\end{remark}

\section*{Acknowledgments}

The authors benefited a lot from participation in \emph{Workshop on Free
Probability and Random Combinatorial Structures}, December 2009, funded
by Sonderforschungsbereich 701 \emph{Spectral Structures and Topological Methods
in Mathematics} at Universit\"at Bielefeld.

Research of PŚ was supported by the Polish Ministry of Higher Education research
grant N N201 364436 for the years 2009--2012.

PŚ thanks Professor Herbert Spohn and his collaborators for their wonderful
hospitality at Technische Universit\"at M\"unchen, where a large part of the
research was conducted. PŚ thanks also Max-Planck-Institut für
extraterrestrische Physik in Garching bei M\"unchen, where a large part of the
research was conducted.

\bibliographystyle{alpha}

\bibliography{biblio2009}

\end{document}